\documentclass[11pt]{amsart}
\usepackage[utf8]{inputenc}
\usepackage{fontenc}
\usepackage{amsfonts}
\usepackage{amssymb}
\usepackage{amsmath}
\usepackage{amsthm}\usepackage{mathtools}
\usepackage{enumerate}
\usepackage{hyperref}\usepackage{enumitem}
\usepackage{mathrsfs}
\usepackage{tikz}\usetikzlibrary{calc}
\usepackage{marginnote}
\usepackage{xcolor,enumitem}
\usepackage{soul}\usepackage{tikz}
\usepackage[all,cmtip]{xy}

\newcommand{\R}{\mathbb{R}}
\newcommand{\Z}{\mathbb{Z}}
\newcommand{\D}{\mathbb{D}}
\newcommand{\M}{\mathbb{M}}
\newcommand{\N}{\mathbb{N}}
\newcommand{\C}{\mathbb{C}}

\newcommand{\cR}{\mathcal{R}}

\newcommand{\cE}{\mathcal{E}}

\newcommand{\eps}{\varepsilon}
\newcommand{\REMARK}[1]{\marginpar{\tt #1}}

\theoremstyle{Case1}

\theoremstyle{Case2}

\newcommand{\broe}{\mathrm{B}_u}
\newcommand{\cstu}{\mathrm{C}^*_u}

\newtheorem*{rigprob*}{Rigidity Problem for uniform Roe Algebras}
\newtheorem*{rigprobcorona*}{Rigidity Problem for uniform Roe Coronas}

\newcommand{\cstar}{$\mathrm{C}^*$}

\newcommand{\cF}{\mathcal{F}}
\newcommand{\cP}{\mathcal{P}}

\newcommand{\cB}{\mathcal{B}}
\newcommand{\cK}{\mathcal{K}}

\newtheorem{theorem}{Theorem}[section]
\newtheorem*{theorem*}{Theorem}
\newtheorem{proposition}[theorem]{Proposition}
\newtheorem{problem}[theorem]{Problem}
\newtheorem*{proposition*}{Proposition}
\newtheorem{lemma}[theorem]{Lemma}
\newtheorem*{lemma*}{Lemma}
\newtheorem{corollary}[theorem]{Corollary}
\newtheorem*{corollary*}{Corollar}

\newtheorem*{fact*}{Fact}
\theoremstyle{definition}
\newtheorem{definition}[theorem]{Definition}
\newtheorem*{definition*}{Definition}
\newtheorem{claim}[theorem]{Claim}
\newtheorem*{claim*}{Claim}

\newtheorem*{conjecture*}{Conjecture}

\newtheorem*{acknowledgments}{Acknowledgments}

\theoremstyle{remark}
\newtheorem{example}[theorem]{Example}
\newtheorem*{example*}{Example}
\newtheorem{remark}[theorem]{Remark}

\newtheorem*{remark*}{Remark}

\newtheorem*{note*}{Note}
\newtheorem*{question*}{Question}

\DeclareMathOperator{\supp}{supp}

\DeclareMathOperator{\propg}{prop}

\DeclareMathOperator{\diam}{diam}

\DeclareMathOperator{\spann}{span}

\newcommand{\cL}{\mathcal L}

\newcounter{my_enumerate_counter}
\newcommand{\pushcounter}{\setcounter{my_enumerate_counter}{\value{enumi}}}
\newcommand{\popcounter}{\setcounter{enumi}{\value{my_enumerate_counter}}}

\usepackage{enumitem}

\begin{document}

\title[On Banach algebras of band-dominated operators]{On Banach algebras	 of band-dominated operators and their order structure}%

\author[B. M. Braga]{Bruno M. Braga}
\address[B. M. Braga]{University of Virginia, 141 Cabell Drive, Kerchof Hall, P.O. Box 400137, Charlottesville, USA}
\email{demendoncabraga@gmail.com}
\urladdr{https://sites.google.com/site/demendoncabraga}

\subjclass[2010]{}
\keywords{}
\thanks{}
\date{\today}%
\maketitle

\begin{abstract}
The goal of this paper is to   study band-dominated operators on Banach spaces with Schauder basis with respect to uniformly locally finite metric spaces as well as  the Banach algebras generated by them: the so called \emph{uniform Roe algebras}. We investigate several kinds of isomorphisms between those Banach spaces (e.g., isomorphisms preserving norm, order, algebraic structure, etc) and prove several rigidity results on when a certain kind of isomorphism between the uniform Roe algebras implies that the base metric spaces  are (bijectively) coarsely equivalent. 
\end{abstract}
 
\setcounter{tocdepth}{1}
\tableofcontents

\section{Introduction}\label{SectionIntro}
The study of band-dominated operators has gathered considerable attention in the last years (e.g., \cite{Chandler-WildeLindner2008JFA,RabinovichRochRoe2004IntEpOpTh,RabinovichRochSilbermann1998IntEqOpTh,RabinovichRochSilbermann2008,Seidel2014LinAlgAppl,SpakulaWillett2017,Willett2011IntEqOpTh}). Those operators have been extensively studied as operators on $\ell_p(\Z^n)$ and later as operators on $\ell_p(X)$, for arbitrary uniformly locally finite metric spaces $(X,d)$.\footnote{Uniformly locally finite metric spaces are often referred to as metric spaces with \emph{bounded geometry} in the literature.} Band-dominated operators form Banach algebras, called \emph{uniform Roe algebras}, and their study has also being boosted   motivated by their connection with the Baum-Connes conjectures \cite{BragaChungLi2019,HigsonRoe1995,Yu2000}, geometric group theory, (higher) index theory \cite{willettYuBook}, and, more recently,   mathematical physics     \cite{EwertMeyer2019,Kubota2017}.

In this paper, we are interested in extending the study of band-dominated operators and their Banach algebras to the broader context of Banach spaces with a fixed Schauder basis $\cE$ --- band-operators are always considered with respect to a fixed  base metric space. We explore how several properties of their Schauder basis affect those Banach algebras. Moreover,  we investigate how   different types of equivalences between those algebras   affect the coarse type of the base metric spaces.

Let $(X,d)$ be a countable metric space and $E$ be a Banach space. As we are interested in dealing  with Schauder basis  which are not necessarily $1$-symmetric,\footnote{A Schauder basis $(e_n)_n$ is \emph{$\lambda$-symmetric} if it is $\lambda$-equivalent to $(e_{\pi(n)})_n$ for all permutations $\pi:\N\to \N$, and it is \emph{symmetric} if it is $\lambda$-symmetric for some $\lambda\geq 1$ (see Section \ref{SectionDifBases} for details). } it is necessary to fix an enumeration of $X$ in order to define band operators on $\cL(E)$ --- where $\cL(E)$ denotes the space of bounded operators on $E$ --- with respect to a basis $\cE$ and  $(X,d)$. Therefore, in order to simplify notation, we assume throughout the paper that $X=\N$ and that $d$ is a metric on $\N$. In this setting, an operator $a\in \cL(E)$ is a  \emph{band operator with respect to $\cE=(e_n)_n$}  if there is $r>0$ so that 
 \[e^*_m(ae_n)=0\ \text{ for all }\ n,m \in \N \text{ with }\ d(n,m)>r,\]
 where $(e^*_m)_m$ is the sequence of biorthogonal functions of $(e_n)_n$, i.e., $e^*_m(e_n)=\delta_{n,m}$ for all $n,m\in\N$. In this case, we say that $a$ has \emph{propagation at most $r$ with respect to $\cE$} and write  $\propg(a)\leq r$.  When the basis $\cE$ is clear from the context, we omit any reference to it.

\begin{definition} 
Consider a metric space $(\N,d)$ and a Banach space $E$ with basis $\cE$. 
\begin{enumerate}
\item The algebra of band operators with respect to $\cE$ is denoted by $\broe[d,\cE]$.
\item An operator $a\in \cL(E)$ is \emph{band-dominated with respect to $\cE$} if it is the norm limit of band operators.
\item The  \emph{uniform Roe algebra of $(\N,d)$ with respect to $\cE$} is the algebra of all band-dominated operators and we denote it by $\broe(d,\cE)$.
\end{enumerate}
\end{definition}

If $E=\ell_2$ and $\cE$ is the standard unit basis of  $\ell_2$, the uniform Roe algebra of $(d,\cE)$ is a \cstar-algebra and it is  usually denoted by $\cstu(\N,d)$, or by $\cstu(X)$  if the metric space under consideration is $(X,d)$. Moreover, for $p\in [1,\infty)$, $X=\ell_p$ and  $\cE$ the standard unit basis of  $\ell_p$,  the uniform Roe algebra of $(d,\cE)$ is often denoted by $\broe^p(\N,d)$, or $\broe^p(X)$. In recent years,  rigidity questions for uniform Roe algebras have been extensively studied for $\ell_p$, with special emphasis for $p=2$ (e.g., \cite{BragaFarah2018,BragaFarahVignati2020,ChungLi2018,SpakulaWillett2013,WhiteWillett2017}). In layman's term, rigidity questions ask what kind of equivalence notions between two  uniform Roe algebras are strong enough so that their existence implies that  the base spaces of those algebras must be (bijectively) coarsely equivalent. This is the central theme of this piece.    The results in this paper hold for both real and complex Banach spaces/algebras. As it is more common in operator algebra theory, we chose to consider all spaces to be complex  unless otherwise stated.\footnote{For this reason, we deal with complexification of real spaces. The reader only interested in real spaces, can simply ignore all complexifications below.} 

We now describe the main findings of this paper. Firstly,  we notice that, for arbitrary Banach spaces with Schauder basis, the space constructed by S. Argyros and R. Haydon as a solution for the scalar-plus compact problem allows us to prove that rigidity cannot work in general. Indeed, there is a Banach space $E$ with a shrinking basis so that $\broe(d,\cE)=\broe(\partial,\cF)$ for all for all metrics $d$ and $\partial$ on $\N$ and all shrinking bases $\cE$ and $\cF$ of $E$ (see Proposition \ref{PropArgyrosHaydon}). 

In order to avoid such pathological examples, we must impose some restrictions on the basis of the Banach space $E$ and symmetry of the basis $\cE$ will be a common running hypothesis throughout these notes.\footnote{Note to the nonspecialist in Banach space theory: the quintessential examples of symmetric spaces are the so called Orlicz sequence spaces    (see \cite[Chapter 3]{LindenstraussTzafririBookVol1}). Notice also that there are Banach spaces with symmetric basis which contain no isomorphic copy of either $\ell_p$ or $c_0$ \cite{FigielJohnson1974}.} Moreover, a fixed basis $\cE=(e_n)_n$ on $E$ defines  natural order structures on $E$ and $\cL(E)$ which play a fundamental role in our results.  Recall, if $E$ is an ordered Banach space, we write     $E_+=\{x\in E\mid x\geq 0\}$ and say that   a  linear map $a:E\to E$ is \emph{positive} if $a(E_+)\subset E_+$ (see Appendix).  If $\cE$ is a Schauder basis for $E$,  an order on $E$ is defined by setting $\sum_na_ne_n\geq 0$ if and only if $a_n\geq 0$ for all $n\in\N$.    This defines  partial orders on both $E$ and $\cL(E)$ which make those spaces into   ordered Banach spaces. Recall, a linear map between ordered Banach spaces is  an \emph{order isomorphism} if it is a positive  bijection whose inverse is also positive.

We show in Theorem \ref{ThmBanachIsomorNorEnough} that there are uniformly locally finite metrics $d$ and $\partial$ on $\N$ so that $\cstu(\N,d)$ and $\cstu(\N,\partial)$ are simultaneously order and Banach space isomorphic, but $d$ and $\partial$ are not coarsely equivalent metrics. However, the next theorem   shows that this cannot happen if we also demand the isomorphism to be norm  preserving.   Recall, a Banach space $E$ is \emph{strictly convex} if $\|x+y\|<2$ for all distinct unit vectors $x,y\in E$.

\begin{theorem}\label{ThmBanachIsometryOrderIso}
Let $d$ and $\partial$ be uniformly locally finite  metrics on $\N$, and  $E$ and $F$ be strictly convex Banach spaces with   $1$-symmetric    bases   $\cE$ and $\cF$, respectively. If  $\broe(d,\cE)$ and $\broe(\partial,\cF)$ are simultaneously order isomorphic and isometric,  then $(\N,d)$ and $(\N,\partial)$ are bijectively coarsely equivalent. 
\end{theorem}

On the other hand, if one forgets about norm preservation, but demands the isomorphism to be an algebra isomorphism instead, then the strong rigidity above still holds:

\begin{theorem}\label{ThmBanachAlgIsomOrderIsoINTRO} 
Let $d$ and $\partial$ be uniformly locally finite  metrics on $\N$, and $E$ and $F$ be Banach spaces with symmetric bases $\cE$ and $\cF$, respectively. If  $\broe(d,\cE)$ and $\broe(\partial,\cF)$ are simultaneously order and Banach algebra isomorphic, then $(\N,d)$ and $(\N,\partial)$ are bijectively coarsely equivalent. 
\end{theorem}

We point out that Theorem \ref{ThmBanachAlgIsomOrderIsoINTRO} can be slightly strengthened in the sense that we do not need the full power of symmetry of the bases, but only a metric version of it (see Definition \ref{DefiDSubsymm} and  Theorem \ref{ThmBanachAlgIsomOrderIso}).

When working with the classic uniform Roe algebras $\cstu(\N,d)$  (or  $\broe^p(\N,d)$),  one does not usually consider the order on $\cstu(\N,d)$ described above, but instead only looks at \cstar-algebra isomorphism (or Banach algebra isomorphism). We generalize several of the currently known results  to arbitrary symmetric basis. The down side of this approach is that, just as in the case $X=\ell_2$,  further geometric assumptions on the metrics $d$ and $\partial$ are necessary in  order to obtain rigidity results (at least with the current techniques).  For that, the definition of ghost operators is necessary: an operator $a\in \cL(E)$ is a \emph{ghost with respect to $\cE$} if for all $\eps>0$ there is $n_0\in\N$ so that $|e^*_m(ae_n)|\leq \eps$ for all $n,m\geq n_0$.\footnote{See Subsection \ref{SubsectionGhosts} for more details on ghost operators.}

\begin{theorem}\label{ThmIsomorBanAlg}
Let $d$ and $\partial$ be uniformly locally finite  metrics on $\N$, and $E$ and $F$ be Banach spaces with symmetric bases $\cE$ and $\cF$, respectively. Assume that all ghost idempotents  in $\broe(d, \cE)$ and $\broe(\partial, \cF)$ are compact. If $\broe(d,\N)$ and  $ \broe(\partial, \N)$ are  Banach algebra isomorphic, then $(\N,d)$ and  $(\N,\partial)$ are coarsely equivalent.
\end{theorem}

As for Theorem \ref{ThmBanachAlgIsomOrderIsoINTRO}, only a metric version of basis' symmetry is needed for Thereom \ref{ThmIsomorBanAlg} (see    Theorem \ref{ThmIsomorBanAlgMetricSym}).

Theorem \ref{ThmIsomorBanAlg} allows us to obtain rigidity results under the assumption that $(\N,d)$ and $(\N,\partial)$ are metric spaces with G. Yu's property A (see Definition \ref{DefiPropA}). For that we introduce the notion of \emph{regular uniform Roe algebra}. Recall, if $\cE$ is $1$-unconditional, then $E$ is a Banach lattice. In this setting,  $\cL(E)_r$ denotes the space of \emph{regular operators} on $E$ and the \emph{regular norm}   $\|\cdot\|_r$ is a norm on $\cL(E)_r$ which makes it into a   Banach lattice (see Appendix and Section \ref{SectionRegularURA} for definitions). Given a uniformly locally finite metric $d$ on $\N$, and a Banach space $E$ with $1$-unconditional basis, we show that all band operators in $\cL(E)$ are regular (see Proposition \ref{PropFinPropImpliesRegular}). Therefore, we can  introduce the notion of \emph{regular uniform Roe algebra}:

\begin{definition}\label{DefiREgURA}
Let $d$ be a uniformly locally finite  metric on $\N$  and $E$ be a Banach space with a $1$-unconditional basis $\cE$. We define the \emph{regular uniform Roe algebra of the pair $(d,\cE)$} as the $\|\cdot\|_r$-closure of all band operators, and denote it by  $\broe^r(d,\cE)$.
\end{definition}

It is well known that if $(\N,d)$ is a uniformly locally finite metric space with  property A, then all ghosts in $\cstu(\N,d)$ are compact (see \cite[Proposition 11.43]{RoeBook}).\footnote{This is actually equivalent to property A for uniformly locally finite metric spaces \cite[Theorem 1.3]{RoeWillett2014}.} The following shows that  the same holds for any $1$-symmetric basis at least for elements in $\broe^r(d,\cE)$.

\begin{theorem}\label{ThmPropertyAGhostsAreComp}
Let $(\N,d)$ be a uniformly locally finite metric space with property A and $E$ be a Banach space with a $1$-unconditional symmetric basis $\cE$. Then all   ghost operators in $\broe^r(d,\cE)$  are compact.
\end{theorem}

Theorem \ref{ThmIsomorBanAlg} and Theorem \ref{ThmPropertyAGhostsAreComp}  allow us to obtain the next result.

\begin{theorem}\label{ThmIsomorBanAlgPropAINTRO}
Let $ d$ and $\partial$ be uniformly locally finite metrics on $\N$,  and $E$ and $F$ be Banach spaces with $1$-unconditional symmetric bases $\cE$ and $\cF$, respectively. If $(\N, \partial)$ has property A and there is a Banach algebra isomorphism $\Phi:\broe(d,\N)\to  \broe(\partial, \N)$ so that $\Phi(\broe^r(d,\cE))=\broe^r(\partial,\cF)$, then $(\N,d)$ and  $(\N,\partial)$ are coarsely equivalent.
\end{theorem}

We point out that the condition  $\Phi(\broe^r(d,\cE))=\broe^r(\partial,\cF)$ in the theorem above can be weakened (as well as the basis' symmetry). For details,  see  Theorem \ref{ThmIsomorBanAlgPropA} below.

The paper is organized as follows: In Section \ref{SectionPrelim}, we introduce the basic definitions and terminologies for this paper --- with the exception of real Banach lattices, whose basic properties we recall  in Appendix. Section \ref{SectionDifBases} deals with several types of basis in Banach spaces (e.g., unconditional, shrinking, symmetric, and subsymmetric) and how those basis' properties manifest in  uniform Roe algebras. 
In Section \ref{SectionRegularURA}, we recall the basics of complex Banach lattices and show that all band operators are regular (see Proposition \ref{PropFinPropImpliesRegular}). Section \ref{SectionRigidityOrder} deals with rigidity results for order isomorphisms between the uniform Roe algebras. We show that ordered Banach space isomorphism is not enough to give us rigidity even when $E=\ell_2$ (see Theorem \ref{ThmBanachIsomorNorEnough}) and we prove Theorem \ref{ThmBanachIsometryOrderIso} and Theorem  \ref{ThmBanachAlgIsomOrderIsoINTRO}. We deal with metric spaces with property A in Section \ref{SectionPropA} and prove Theorem \ref{ThmPropertyAGhostsAreComp}.
The rigidity results for Banach algebra isomorphisms are obtained in Section \ref{SectionRigBanAlgIso}. Moreover, besides the isomorphism results mentioned above, we also obtain certain embedding results (see   Theorem \ref{ThmEmbHer}  and Theorem \ref{ThmEmbHerPropA}). In Section \ref{SectionOpenProb}, we  list some natural problems which are left open in this paper and we finish the paper with a short appendix on real Banach algebras.

\section{Preliminaries}\label{SectionPrelim}

\subsection{Basic terminology on Banach spaces}
The results on this paper hold for both real and complex Banach spaces. As a rule, we chose to follow the common approach used in operator algebra theory, and all the  Banach spaces are assumed to be over the complex field unless otherwise stated. We refer to \cite{LindenstraussTzafririBookVol1} for the basics of Banach space theory.  

Given  Banach spaces $E$ and $F$, $B_E$ denotes the  closed unit ball of $E$  and the space of all bounded operators $a:E\to F$ is denoted by $\cL(E,F)$, if $E=F$ we simply write $\cL(E)$. The  space  of all compact operators in  $\cL(E,F)$   is denoted by $\cK(E,F)$ (if $E=F$, we write $\cK(E)$).

A sequence $(e_n)_n$ in a Banach space $E$ is a \emph{Schauder basis for $E$}, or simply a \emph{basis for $X$}, if each $x\in E$ can be written uniquely as $x=\sum_{n\in\N}\alpha_ne_n$ for some $(\alpha_n)_n\in \C^\N$. If $P_n:E\to \spann\{e_1,\ldots,e_n\}$ denotes the standard projection, then $\sup_n\|P_n\|<\infty$. The basis is called \emph{monotone} if $\sup_n\|P_n\|=1$ and \emph{bimonotone} if it is monotone and $\sup_n\|\mathrm{Id}_E-P_n\|=1$.  A sequence $(x_n)_n$ in $E$ is called a \emph{basic sequence} if it is a basis for the closure of its span. Given $\lambda\geq 1$, two sequences $(x_n)_n $ and $(y_n)_n$ in $E$ are called \emph{$\lambda$-equivalent} if 
\[\frac{1}{\lambda}\Big\|\sum_{n=1}^k\alpha_nx_n\Big\|\leq \Big\|\sum_{n=1}^k\alpha_ny_n\Big\|\leq \lambda \Big\|\sum_{n=1}^k\alpha_nx_n\Big\|\]
for all $k\in\N$ and all $(\alpha_n)_n\in \C^k$. 

Given a sequence $(x_n)_n$ in a Banach space $E$, we say that $(x_n)_n$ is  \emph{semi-normalized} if it is bounded and  $\inf_{n}\|x_n\|>0$. Given another sequence   $(y_n)_n$ in  $E$, we say that $(y_n)_n$ is a \emph{block subsequence of $(x_n)_n$} if there are a strictly increasing sequence $(n_k)_k\in \N^\N$ and $(\alpha_k)_k\in \C^\N$ so that $y_k=\sum_{i=n_k}^{n_{k+1}-1}\alpha_ix_{i}$ for all $k\in\N$.
 
\begin{remark}
Let $d$ me a metric on $\N$,  $(E,\|\cdot\|)$ be a Banach space with a basis $\cE$, and let $\|\cdot\|'$ be  a norm on  $E$ equivalent to $\|\cdot\|$. Consider $\cE'=(e_n/\|e_n\|)_n$ as a basis for $(E,\|\cdot\|')$. Notice that $\broe[d,\cE]=\broe[d, \cE']$, so $\broe(d,\cE)=\broe(d, \cE')$. This simple observation will be used throughout.
\end{remark}

Throughout these notes, given finitely many Banach spaces $(E_i)_{i=1}^k$, we denote their $\ell_\infty$-sum  by  $\bigoplus_{i\in I}E_i$, i.e., the norm of an $I$-tuple $x=(x_i)_{i\in I}$ is given by $\|x\|=\max\{\|x_i\|\mid i\in I\}$.

\subsection{Metric spaces and coarse equivalences}

Let $f:(X,d)\to (Y,\partial)$ be a map. We say that $f$ is \emph{coarse} if for all $r>0$ there exists $s>0$ so that 
\[d(x,y)<r\ \text{ implies }\ \partial(f(x),f(y))<s\]
for all $x,y\in X$. The map $f$ is called \emph{expanding} if for all $s>0$ there exists $r>0$ so that 
\[d(x,y)>r\ \text{ implies } \ \partial(f(x),f(y))>s\]
for all $x,y\in X$. If $f$ is both coarse and expanding, $f$ is called a \emph{coarse embedding}. We say that $f$ is a \emph{coarse equivalence} if $f$ is a coarse embedding which is also \emph{cobounded}, i.e., $\sup_{y\in Y}\partial (y,f(X))<\infty$. In this case, we say that $(X,d)$ and $(Y,\partial)$ are \emph{coarsely equivalent}. Equivalently, $(X,d)$ and $(Y,\partial)$ are coarsely equivalent if there are coarse maps $f:X\to Y$ and $g:Y\to X$ so that $f\circ g $ and $g\circ f$ are close to $\mathrm{Id}_{Y}$ and $\mathrm{Id}_{X}$, respectively.\footnote{Recall, maps $f,h:X\to (Y,\partial)$ are \emph{close} if $\sup_{x\in X}\partial(f(x),h(x))<\infty$.}

A metric space $(X,d)$ is \emph{locally finite}  if $|B_r(x)|<\infty$ for all $r>0$ and all $x\in X$, where $|B_r(x)|$ denotes the cardinality of the closed $d$-ball centered at $x$ of radius $r$. We say that $(X,d)$ is   \emph{uniformly locally finite}, abbreviated by \emph{u.l.f.}, if $\sup_{x\in X}|B_r(x)|<\infty$ for all $r>0$. Clearly, a locally finite metric space must be countable. Hence, fixing an enumeration for an infinite  $X$, it is enough to deal with locally finite metrics defined on $\N$.

The following is well known and its proof   can be found for instance in \cite[Proposition 2.4]{BragaFarah2018}.

\begin{lemma}\label{LemmaULFPartition}
Let $(\N,d)$ be a u.l.f. metric space and $r>0$. Then there exists a partition  
\[\{(n,m)\in \N\times \N\mid d(n,m)\leq r\}=A_1\sqcup A_2\sqcup \ldots \sqcup A_k\]
such that  each $A_i$ is the graph of a partial bijection\footnote{A \emph{partial bijection} of $\N$ is a bijection between subsets of $\N$.} of $\N$. Moreover, the partion can be taken so that  $d(n_1,n_2)>r$ and $d(m_1,m_2)>r$ for all $i\in \{1,\ldots, k\}$ and all distinct $(n_1,m_1),(n_2,m_2)\in A_i$.
\end{lemma}

Given a Banach space $E$ with basis $\cE$ and an operator $a\in \cL(E)$, we define the \emph{support of $a$} by 
\[\supp(a)=\{(n,m)\in \N\times \N\mid e^*_m(ae_n)\neq 0\}.\]
The previous lemma gives a way to split the support of a finite propagation operator into finitely many pieces, and, if $\cE$ is unconditional, we can write such operators as the finite sum of operators supported by members of this partition (e.g., proof of Proposition \ref{PropGhostWithFInProgAreComp}).

\subsection{Matrix representation and the matrix algebra  $\mathrm M_\infty(\cE)$}\label{SubsectionGhosts}

Given $n,m\in\N$, we write $e_{n,m}$ to denote the infinite $\N$-by-$\N$ matrix so that all of its coordinates are zero except the $(m,n)$ coordinate, which equals 1.  If   $E$ is a Banach space with   basis $\cE=(e_n)_n$, the matrices $(e_{n,m})_n$ induce bounded operators on $E$ in a canonical way: $e_{n,m}e_n=e_m$ for all $n,m\in\N$. Moreover, if $(e_n)_n$ is normalized, then $\|e_{n,m}\|=1$ for all $n,m\in\N$.  If $A\subset \N$, we write  $\chi_A=\sum_{n\in\N}e_{n,n}$.  For $n\in\N$, we write $\chi_{[1,n]}=\chi_{[1,n]\cap \N}$ and $\chi_{[n,\infty)}=\chi_{[n,\infty)\cap \N}$.

Let $E$ be a Banach space with basis $\cE=(e_n)_n$. For each $n\in\N$, we let $ \mathrm M_n(\cE)$ denote the algebra of $n$-by-$n$ matrices operators on $E$. Precisely, 
\[\mathrm M_n(\cE)=\cL(\spann\{e_j\mid j\leq n\}).\]
We view each $\mathrm M_n(\cE)$ as a subalgebra of $\cL(E)$ in  the natural way, i.e., we make the following identification:
\[\mathrm M_n(\cE)=\{a\in \cL(E)\mid e^*_m(ae_k)=0\ \text{ for all }\ m,k>n\}.\]
 Under this identification,\footnote{If $\cE$ is monotone, the identification $\mathrm M_n(\cE)\subset \cL(E)$  is also an isometry. Notice that renorming $E$ so that $\cE$ is  monotone does not affect $\mathrm M_\infty(\cE)$.} we let 
\[\mathrm M_\infty(\cE)=\overline{ \bigcup_{n\in\N}\mathrm M_n(\cE)}^{\cL(E)}.\]

\begin{proposition}
Let $(\N,d)$ be a locally finite metric space and $\cE$ be a basis for a Banach space $E$. Then 
\[\mathrm M_\infty(\cE)=\broe(d,\cE)\cap \cK(E).\]
\end{proposition}

\begin{proof}
The inclusion $ \mathrm M_\infty(\cE)\subset \broe(d,\cE)\cap \cK(E)$ is clear. For the other inclusion, let $a\in\broe(d,\cE)\cap \cK(E)$ and assume that $a\not\in \mathrm M_\infty(\cE)$. Then there exists $(\xi_n)_n\in B_E$ so that $\inf_n\|a\xi_n\|>0$ and either $\xi_n\in \spann\{e_j\mid j>n\}$ for all $n\in\N$ or $a\xi_n\in \spann\{e_j\mid j>n\}$ for all $n\in\N$. Since $a$ is compact, the latter   cannot happen, so assume that $\xi_n\in \spann\{e_j\mid j>n\}$ for all $n\in\N$. Since $a$ is compact, by going to a subsequence, we can assume that there exists $m\in\N$ so that $\inf_n\|e^*_m(a\xi_n)\|>0$. Since $(\N,d)$ is locally finite, this contradicts the assumption that $a$ can be approximated by finite propagation operators.
\end{proof}

Given a u.l.f. metric space $(\N,d)$, it easily follows that the space of all ghosts operators in $\broe(d,\cE)$ is an ideal of  $\broe(d,\cE)$ (see Section \ref{SectionIntro} for definition of ghosts). Evidently, the ideal of all ghosts always contains $\mathrm M_\infty(\cE)$.  In general, this is a strict inclusion  (see \cite[Page 349]{higson2002counterexamples} or \cite[Theorem 1.3]{RoeWillett2014}). However, under the assumption that the metric space has G. Yu's property A, we can often guarantee that ``reasonable'' ghosts must belong to $\mathrm M_\infty(\cE)$ (see Theorem \ref{ThmPropertyAGhostsAreComp} for a precise statement).

\begin{proposition}\label{PropGhostWithFInProgAreComp}
Let $(\N,d)$ be a u.l.f. metric space and $E$ be a Banach space with a symmetric basis $\cE$. Then all ghost operators with finite propagation belong to $\mathrm{M}_\infty(\cE)$. 
\end{proposition}

\begin{proof}
Let $a$ be ghost with   propagation at most $r$. Let 
\[\{(n,m)\in \N\times \N\mid d(n,m)\leq r\}=A_1\sqcup\ldots \sqcup A_k\] be the partition given by Lemma \ref{LemmaULFPartition} for $r$. Unconditionality of $\cE$ allow us to write $a=\sum_{i=1}^ka_i$ where each $a_i$ has its support contained in $A_i$. As $a$ is a ghost, so is each $a_i$. Hence, symmetry of the basis implies that $a_i\in \mathrm{M}_\infty(\cE)$ for all $i\in \{1,\ldots, k\}$.
\end{proof}

\section{Properties of different bases in uniform Roe algebras}\label{SectionDifBases} In this section, we deal with how different basis' properties affect the uniform Roe algebras.  Recall, given a basis $\cE=(e_n)_n$ for a Banach space $E$ and $\lambda\geq 1$, the basis $\cE=(e_n)_n$ is called \emph{$\lambda$-unconditional} if \[\Big\|\sum_{n=1}^k\alpha_ne_n\Big\|\leq \lambda \Big\|\sum_{n=1}^k\beta_ne_n\Big\|\] for all $k\in\N$ and all $(\alpha_n)_{n=1}^k,(\beta_n)_{n=1}^k\in \C^k$ with $|\alpha_n|\leq |\beta_n|$ for all $n\in\{1,\ldots, k\}$. The smallest such $\lambda$ is called the \emph{unconditional constant of $\cE$}. A basis $\cE$ is \emph{unconditional} if it is $\lambda$-unconditional for some $\lambda\geq 1$. 

Throughout the paper, we consider
\[\ell_\infty=\Big\{a=[a_{i,j}]\in\C^{\N\times \N} \mid \sup_{i,j}|a_{ij}|<\infty\ \text{ and }\  a_{ij}= 0,\ \forall i\neq j\Big\}\]
and  we  view each element in $\ell_\infty$ as a linear map on $\C^\N$ by setting $a \xi=(a_{ii}\xi_i)_i$ for each $\xi=(\xi_i)_i\in \C^\N$. The proof of the next  proposition is completely elementary,  so we omit it.  

\begin{proposition}\label{PropUnconditional}
Consider a  metric space $(\N,d)$ and a Banach space $E$ with basis $\cE$. Then  $\cE$ is unconditional if and only only if $\ell_\infty\subset \broe(d,\cE)$, i.e., if each  linear linear map in $\ell_\infty$ defines a bounded operator on $E$.\qed
\end{proposition}

In the setting of the proposition above, we write $\ell_\infty(\cE)$ to denote the vector space $\ell_\infty$ endowed with the norm inherited by $\broe(d,\cE)$.

As mentioned in the introduction, if $(X,d)$ is a u.l.f. metric space and $\cE$ is the standard basis of $\ell_p$, the uniform Roe algebra $\broe(d,\cE)$ is frequently denoted by $\broe^p(X)$ (see \cite{BragaVignati2019,ChungLi2018,SpakulaWillett2017}). In this case,   $\mathrm{B}^p_u(X)$   contains the compact operators if and only if $p\neq 1$. Indeed, if $f\in \ell_\infty$ is the functional on $\ell_1$ so that $f(e_n)=1$ for all $n\in\N$, then $e_1\otimes f$ is compact but its  distance to $\mathrm{B}^1_u(X)$ is $1$. The next proposition generalizes this fact to a broader setting. Recall, a basis $\cE=(e_n)_n$ for a Banach space $E$ is called \emph{shrinking} if the sequence of its biorthogonal functionals $(e_n^*)_n$ is a basis for the dual Banach space $E^*$.

\begin{proposition}\label{PropShrinkingCompact}
Consider a locally finite metric space $(\N,d)$ and a Banach space $E$ with basis $\cE$.  Then $\broe(d,\cE)$ contains the compact operators if and only if $\cE$ is shrinking. 
\end{proposition}

\begin{proof}
Suppose $\cE=(e_n)_n$ is shrinking. Since compact operators are in the closure of the set of finite rank operators for any Banach space with the approximation property, it is enough to notice that $\broe(d,\cE)$ contains the finite rank operators. Let $a$ be a finite rank operator, so $a=\sum_{i=1}^k x_i\otimes f_i$ for some $x_1,\ldots, x_k\in E$ and some $f_1,\ldots, f_k\in E^*$. Let $P_k:E\to E$ and $Q_k:E^*\to E^*$ be the projections onto the subspaces of $E$ and $E^*$ generated by the first $k$ elements of $(e_n)_n$ and $(e^*_n)_n$, respectively. Picking $m\in\N$ large enough, we have that $\sum_{i=1}^k P_m(x_i)\otimes Q_m( f_i)$ is as close to $a$ as we wish. This concludes that $\cK(E)\subset \broe(d,\cE)$.

Suppose $\cE=(e_n)_n$ is not shrinking, and let $f\in E^*$ be a functional so that $\delta=\limsup_m\|f\|_m>0$, where
\[\|f\|_m\coloneqq \{|f(x)|\mid x\in B_{\spann\{e_n\mid n\geq m\}}\}\]
for each $m\in\N$. Without loss of generality, assume that $\cE$ is normalized. Then $a=e_1\otimes f$ is a rank 1 operator which does not belong to $\broe(d,\cE)$. Indeed, let $b\in \broe[d,\cE]$. Then, since $(\N,d)$ is locally finite, there exists $m\in\N$ so that $d(1,n)>\propg(b)$ for all $n>m$. Replacing $m$ by a larger number if necessary, assume that $\|f\|_m>\delta/2$. Pick $x\in B_{\spann\{e_n\mid n\geq m\}}$ 
so that $|f(x)|\geq \delta/2$. Then $e^*_{1}b(x)=0$, and we conclude that
\[\frac{\delta}{2}\leq |f(x)|=|e^*_{1}(a-b)(x)|\leq \|e^*_{1}(a-b)\|.\]
So $a\not\in \overline{\broe[d,\cE]}=\broe(d,\cE)$.\end{proof}

A basic sequence $\cE=(e_n)_n$ is \emph{symmetric} if $(e_n)_n$ is equivalent to $(e_{\pi(n)})_n$   for all bijections  $\pi:\N\to\N$, and it is well known that symmetric basis are automatically unconditional. Moreover, we say that $\cE$ is \emph{$1$-symmetric} if $(e_n)_n$ is both  $1$-unconditional and  $1$-equivalent to $(e_{\pi(n)})_n$   for all bijections $\pi:\N\to\N$.  

\begin{remark}
Notice that, if $\cE=(e_n)_n$ is $1$-symmetric, then there is no need to fix an enumeration of a countable metric space $(X,d)$ in order to define band operators. Therefore, for $1$-symmetric basis, one could should to use the (more common) notation $\broe(X,\cE)$ to denote the uniform Roe algebra of $(d,\cE)$.  
\end{remark}

We now introduce a (nonstandard) metric version of basis' symmetry which gives a characterization of operators in $\broe[d,\cE]$ and will be  very useful later --- this characterization  is also heavily used in the classic $\ell_2$ scenario.

\begin{definition}\label{DefiDSubsymm}
Let $(\N,d)$  be a metric space, and  $\cE$ be an unconditional basis for  a Banach space $E$. We say that $\cE$ is \emph{$d$-symmetric} if for all $r>0$ and all partial bijections $\sigma:A\subset \N\to B\subset \N$ so that  $d(i,\sigma(i))\leq r$ for all $i\in A$, there exists $C>0$ so that $(e_n)_{n\in A}$ is $C$-equivalent to $(e_{\sigma(n)})_{n\in B}$. 
\end{definition}

Clearly, any symmetric sequence is $d$-symmetric regardless of the metric $d$ on $\N$.

\begin{proposition}\label{PropSubsymmetric}
Consider a u.l.f. metric space $(\N,d)$ and a Banach space $E$ with   basis $\cE$. Then the following are equivalent:
\begin{enumerate}
\item\label{ItemSubsymmetric1} The basic sequence $\cE$ is $d$-symmetric.
\item\label{ItemSubsymmetric2} For all  $r>0$ there exists $M>0$ so that for every  matrix $a=[a_{n,m}]\in \C^{\N\times\N}$ with propagation at most $r$, the operator norm of $a$ satisfies
\[\|a\|\leq M\sup_{n,m}|a_{n,m}|.\]
\end{enumerate}
In particular, \eqref{ItemSubsymmetric2} holds if $\cE$ is symmetric.
\end{proposition}

\begin{proof}
\eqref{ItemSubsymmetric1} $\Rightarrow $ \eqref{ItemSubsymmetric2} Fix $r>0$.  Since $d$ is u.l.f., let
\[\{(n,m)\in \N\times \N\mid d(n,m)\leq r\}= A_1 \sqcup A_2 \sqcup \ldots \sqcup A_k\]
be the partition given by Lemma \ref{LemmaULFPartition} for $r$. Let $C>0$ witnesses that $\cE$ is $d$-symmetric with relation  to $A_1,\ldots, A_k$.  Fix  $a=[a_{n,m}]$ with finite propagation at most $r$.  Since $\cE$ is unconditional, we can write $a=\sum_{i=1}^ka_k$, where each $a_i$ is a bounded operator so that  $\supp(a_i)\subset A_i$ for all $i\in \{1,\ldots, k\}$. 

Fix $i\in \{1,\ldots, k\}$, let $\sigma:A\subset \N\to B\subset \N$ be a partial bijection on $\N$ so that $A_i=\mathrm{graph}(\sigma)$, and  $a_i=[a^i_{n,m}]$. Then
\[a_i\Big( \sum_{n=1}^\infty \alpha_ne_n\Big)=\sum_{n=1}^\infty a^i_{n\sigma(n)}\alpha_ne_{\sigma(n)}\]
for all $\sum_{n=1}^\infty \alpha_ne_n\in E$. Let $L$ be the   unconditional constant of $\cE$. It follows that $\|a_i\|\leq LC\sup_{n,m}|a^j_{n,m}|$. Hence, $\|a\|\leq kLC\sup_{n,m}|a_{n,m}|$, and we are done.

\eqref{ItemSubsymmetric2} $\Rightarrow $ \eqref{ItemSubsymmetric1} This follows  by looking at matrices $a$ whose support are partial bijections $\sigma:A\subset \N\to B\subset\N$ so that $d(n,\sigma(n))\leq r$ for all $n\in A$. We leave the details to the reader.
\end{proof}

\begin{remark}
Notice that  given  a $1$-unconditional basis $\cE$ for a Banach space $E$ and a u.l.f. metric space $(\N,d)$, the assumption $\sup_{i,j}|a_{ij}|<\infty$ is not enough to guarantee that a finite propagation matrix $a=[a_{ij}]$ defines a bounded operator on $E$. Indeed, let $(x_n)_n$ and $(y_n)_n$ denote the standard basis of $\ell_1$ and $\ell_2$,  respectively, and let $e_{2n}=x_n$ and $e_{2n-1}=y_n$ for each $n\in\N$. So $ (e_n)_n$ is a $1$-unconditional basis for $E=\ell_1\oplus \ell_2$. Define $a=[a_{ij}]$ by letting $a_{ij}=1$ if $i=2n-1$ and $j=2n$ for some $n\in\N$ and $a_{ij}=0$ otherwise. Considering $\N$ with its usual metric, $a$ has propagation $1$. However, $a$ is clearly not even a well defined operator. 
\end{remark}

\subsection{Pathological examples}
We start by noticing that  uniform Roe algebra of nonisomorphic Banach spaces can be isomorphic as Banach spaces.

\begin{proposition}\label{PropAlvaroFarmer}
Let $p\in (1,\infty)$. If $d$ and $\partial$ are bounded metrics on $\N$, $\cE$ is a basis for $\ell_p$ and $\cF$ is a basis for $L_p$, then $\broe(d,\cE)$ and $\broe(d,\cF)$ are isomorphic as Banach spaces.
\end{proposition}

\begin{proof}
As the metrics are bounded, it follows that $\broe(d,\cE)=\cL(E)$ and $\broe(\partial,\cF)=\cL(F)$. Hence, the result   follows from \cite[Theorem 2.1]{AlvaroFarmer1996}.
\end{proof}

Notice that $\cL(E)$ and $\cL(F)$ are never Banach algebra isomorphic if $E$ and $F$ are nonisomorphic as Banach spaces 
\cite[Theorem 2]{Eidelheit1940}. Hence, we cannot use the same approach as in Proposition \ref{PropAlvaroFarmer} in order to find nonisomorphic Banach spaces $E$ and $F$ with Banach algebra isomorphic uniform Roe algebras (Problem \ref{Prob4}).

The scalar-plus compact problem was solved in \cite{ArgyrosHaydon2011}, where S. Argyros and R. Haydon constructed  a real Banach space $\mathfrak X$ which is a predual of $\ell_1(\R)$ and  so that all bounded operators on it are a multiple of the identity plus a compact. As such, $\mathfrak X$ admits a shrinking basis, and it is an interesting source for many pathological examples. For instance, we have the following --- since we chose to deal with complex Banach spaces in this paper, we do the same below, but the  real version of the proposition also holds.

\begin{proposition}\label{PropArgyrosHaydon}
Let $\mathfrak X_\C$ be the standard Banach space complexification of the Argyros-Haydon space. Then $\mathfrak X_\C$ has a shrinking basis and the following holds:
\begin{enumerate}
    \item\label{ItemPropArgyrosHaydon1} $\broe(d,\cE)= \broe(\partial, \cF)$ for all metrics $d$ and $\partial $ on $\N$, and all shrinking bases $\cE$ and $\cF$ of $\mathfrak X_\C$.
\item\label{ItemPropArgyrosHaydon2} $\broe(d,\cE)$ is amenable and separable for all  metrics $d$  on $\N$ and all shrinking basis $\cE$ of $\mathfrak X_\C$. In particular, $\broe(d,\cE)$ is not homeomorphic to any uniform Roe algebra which is given by an unconditional basis.
\end{enumerate} 
\end{proposition}

\begin{proof}
Let $\mathfrak X$ be the Argyros-Haydon space, i.e., $\mathfrak{X}$ is a real Banach space which was precisely constructed so that all real-linear bounded operators on $\mathfrak X$ are of the form $\alpha\cdot\mathrm{Id}_\mathfrak{X}+K$ for some $\alpha\in \R$ and some real-linear compact operator $K$ on $\mathfrak X$  (see \cite{ArgyrosHaydon2011}). Let $\mathfrak{X}_\C$  be the standard Banach space complexification of $\mathfrak{X}$, i.e, $\mathfrak{X}_\C$ is real-linear isomorphic to $\mathfrak{X}\oplus \mathfrak{X}$, $(\alpha+i\beta)(x,y)=(\alpha x-\beta y,\beta x+\alpha y)$, and 
\[\|(x,y)\|=\sup_{t\in [0,2\pi]}\|\cos(t)x+\sin(t)y\|\] for all $x,y\in \mathfrak X$ and all $\alpha,\beta\in \R$ (see \cite{MunozSarantopoulosAndrew1999Studia} for details on this complexification). As $\mathfrak{X}$ is an $\ell_1(\R)$-predual, $\mathfrak{X}_\C$ is an $\ell_1(\C)$ predual (see \cite[Proposition 7]{MunozSarantopoulosAndrew1999Studia}). In particular,  $\mathfrak{X}_\C^*$ has a basis, which implies that $ \mathfrak{X}_\C$ has a shrinking basis (see \cite[Theorem 1.4]{JohnsonRosenthalZippin1971}).

\eqref{ItemPropArgyrosHaydon1}  It follows straightforwardly that all operators on  $\mathfrak{X}_\C$ are also a (complex) scalar multiple of the identity plus a compact,  so $\cL(\mathfrak X_\C)=\cK(\mathfrak X_\C)+\C \cdot \mathrm{Id}_{\mathfrak{X}_\C}$. Hence, $\broe(d,\cE)=\cL(\mathfrak{X}_\C)$  for all metric $d$ on $\N$, and all shrinking bases $\cE$ of  $\mathfrak X_\C$. Indeed, $\C \cdot \mathrm{Id}_{\mathfrak{X}_\C}\subset \broe(d,\cE)$ always holds and Proposition \ref{PropShrinkingCompact} implies that $\cK(\mathfrak X_\C)\subset \broe(d,\cE)$. 

\eqref{ItemPropArgyrosHaydon2} Separability  follows since $\cL(\mathfrak X_\C)=\cK(\mathfrak X_\C)+\C \cdot \mathrm{Id}_{\mathfrak{X}_\C}$ and amenability follows from \cite[Proposition 10.6]{ArgyrosHaydon2011}. For the  last statement, notice that   Proposition \ref{PropUnconditional} implies that $\broe(\partial,\cF)$ it not separable for all unconditional basis $\cF$.
\end{proof}

\section{Regular uniform Roe algebra  }\label{SectionRegularURA}

In order to avoid anomalies as   Proposition \ref{PropArgyrosHaydon}, we will restrict ourselves to Banach spaces with  unconditional basis. One of the main benefits of that is the fact that Banach spaces with $1$-unconditional basis are Banach lattices --- very simple Banach lattices. In this subsection, we use this lattice structure in order to define the regular uniform Roe algebra and we show that operators of finite propagation are regular. We start this section with a quick review of complex Banach lattices and  refer the reader to Appendix  for a review of real Banach lattices, as well as some basic results which we use throughout these notes. 
 
Let $F$ be a real Banach lattice.  Then $F_\C$ denotes the sum $F\oplus F$ and we write $x+iy$ for  an element $(x,y)\in F_\C$. We endow $F_\C$ with the modulus 
\[|x+iy|=\sup_{t\in [0,2\pi]}|\cos(t)x+\sin(t)y|,\]
where the supremum above is taken in $F$, and endow $F_\C$ with the norm \[\|x+iy\|=\||x+iz|\|,\]
for all $x+iy\in F_\C$ (see \cite[Section II.11]{SchaeferBook1974} for details). The space $F_\C$ is called a \emph{complex Banach lattice} --- and the \emph{lattice complexification of $F$}. Clearly, every (complex) Banach space $E$ with a $1$-unconditional basis is a complex Banach lattice. Given an operator $a\in \cL(F_\C)$, there are operators $a_r,a_c\in \cL(F)$ so that \[a(x)=a_r(x)+ia_c(x)\] for all $x\in F$. Those operators are the \emph{real} and \emph{complex parts} of $a$, respectively. We say that $a$ is \emph{positive} if $a_r$ is (real) positive and $a_c=0$,  and we say that $a$ is \emph{regular} if both $a_r$ and $a_c$ are (real) regular  (see  Appendix). We denote the positive operators on $F_\C$ by $\cL(F_\C)_+$ and the set of regular operators by $\cL(F_\C)_r$. 
 
\begin{proposition} \cite[Proposition 2.2.6]{Meyer-NiebergBook1991}
Let $F$ be a Dedekind complete real Banach lattice. Then $|a\xi|\leq |a||\xi|$ for all $a\in \cL(F_\C)_r$ and all $\xi\in F_\C$.\label{PropDedekindBound}
\end{proposition}

 The next proposition shows that the elements of $\broe[d,\cE]$ are all regular if $\cE$ is $1$-unconditional. This will be essential for Section \ref{SectionPropA} in the proofs of Lemma \ref{LemmaSeriesSOTconv} and Lemma \ref{LemmaSeriesSOTconvCommutator}.

\begin{proposition}\label{PropFinPropImpliesRegular}
Let $(\N,d)$ be a u.l.f. metric space and $E$ be a Banach space with an $1$-unconditional basis $\cE$. Then $\broe[d,\cE]\subset \cL(E)_r$.
\end{proposition}

\begin{proof}
Let $a\in \broe[d,\cE]$. Write $a=b_1-b_2 +i(b_3-b_4)$ where each $b_i$ has only positive entries and all of their nonzero entries  belong to $\supp(a)$.  Since all positive operators on a real Banach lattice are bounded (Proposition \ref{PropPositiveAreBounded}), it is enough to show that each $b_i$ is a well defined positive linear operator.

Let $b=b_1$. Since $(\N,d)$ is a u.l.f. and $\propg(b)\leq \propg(a)<\infty$, there exists a partition 
\[\supp(b)=C_1\sqcup\ldots \sqcup C_k\]
so that each $C_i$ is the graph of a partial bijection, say $C_i=\{(x_i^\ell,y_i^\ell)\}_\ell$, and $B_r(y^\ell_i)\cap B_r(y^m_i)=\emptyset$ for all $i\in \{1,\ldots, k\}$ and all  $\ell\neq m$ (Lemma \ref{LemmaULFPartition}).  For each $i\in \{1,\ldots, k\}$, let $c_i$ be the restriction of $b$ to $C_i$, i.e., $\supp(c_i)=C_i$ for all $i\in \{1,\ldots, k\}$ and $b=\sum_{i=1}^kc_i$.  

In order to show that $b$ is a well defined positive linear operator, it is enough that the same holds for each $c_i$. Fix $i\in\{1,\ldots, k\}$, and let $c=c_i$ and $C=C_i$. To simplify notation, let $C=\{(x_n,y_n)\in \N\times \N\mid n\in\N\}$.  Let us observe that $c\xi$ is well defined for each $\xi\in E$. Indeed, this follows since, for any $\xi\in E$ with finite support $c\xi$ is well defined and we have that 
\[c\xi\leq b(\xi_{\{x_n\mid n\in\N\}})\]
(here we use that all the entries of $b$ are positive and that $B_r(y^\ell_i)\cap B_r(y^m_i)=\emptyset$ for all $i\in \{1,\ldots, k\}$). Similarly, we get have that 
\[b(\xi_{\{x_n\mid n\in\N\}})\leq |a(\xi_{\{x_n\mid n\in\N\}})|.\]
Moreover, since $\cE$ is $1$-unconditional, $\|\zeta\|=\||\zeta|\|$ for all $\zeta\in E$. Hence, his shows that 
\[\|c\xi\|\leq \|a(\xi_{\{x_n\mid n\in\N\}})\|\leq \|a\|\|\xi_{\{x_n\mid n\in\N\}}\|\leq \|a\|\|\xi\|.\]
Since $\xi\in E$ is an arbitrary element of $E$ with finite support, this shows that $c\xi$ is a well defined element of $E$ for all $\xi\in E$. Since $c$ is clearly linear and positive, Proposition \ref{PropPositiveAreBounded} implies that $c\in \cL(E)$. 

Since $i\in \{1,\ldots,k\}$ was arbitrary, this shows that $b\in \cL(E)_r$. Analogously, this argument shows that $b_2,b_3,b_4\in \cL(E)_r$, and we conclude that $\broe[d,\cE]\subset \cL(E)_r$. 
\end{proof}

If $F$ is a real Banach lattice, then $\cL(F_\C)$ is isomorphic to $\cL(F)_\C$. If moreover, $F$ is Dedekind complete, then every regular operator $a\in \cL(F_\C)$ has a well defined modulus $|a|$ given by 
\[|a|=\sup_{t\in [0,2\pi]}|\cos(t)a_r+\sin(t)a_c|\]
(see \cite[Theorem 1.8]{SchaeferBook1974}). Moreover, the regular operators $\cL(F_\C)_r$ can be endowed with the \emph{regular norm} (or \emph{$r$-norm}) given by 
\[\|a\|_r=\||a|\|\]
for all $a\in \cL(F_\C)_r$. Hence,  $(\cL(F_\C)_r, \|\cdot\|_r)$ is a (complex) Banach lattice (see \cite[Corollary IV.1.1]{SchaeferBook1974}).

Let $E$ be a (complex) Banach space with a $1$-unconditional basis $\cE$. Since $\broe[d,\cE]\subset \cL(E)_r$ (Proposition \ref{PropFinPropImpliesRegular}), one can return to the definition of the uniform Roe algebra of the pair $(d,\cE)$ and take the norm closure  of $\broe[d,\cE]$ with respect to the regular norm $\|\cdot\|_r$. This defines the regular uniform Roe algebra of $(d,\cE)$,  $\broe^r(d,\cE)$, which we defined in Section \ref{SectionIntro} (see Definition \ref{DefiREgURA}). Notice that, in general, we have the following inclusions:
\[\broe[d,\cE]\subset \broe^r(d,\cE)\subset \broe(d,\cE)\cap \cL(E)_r.\]
Also, the lattice structure of $\cL(E)_r$ makes $\broe^r(d,\cE)$ into a Banach lattice as well:

\begin{proposition}
Let $(\N,d)$ be a u.l.f. metric space, and $E$ be a Banach space with an $1$-unconditional basis $\cE$. Then $\broe^r(d,\cE)$ is a Banach sublattice of $\cL(E)_r$.
\end{proposition}

\begin{proof}
As $E$ is a (complex) Banach space with $1$-unconditinal basis, there exists a real Banach space $F$ with an unconditional basis so that $E=F_\C$. By abuse of notation, we still denote the $1$-unconditional basis of $F$ by $\cE$. Let  $\mathrm{B}_{u,\R}^r(d,\cE)$ denote the  $\|\cdot\|_r$-closure of all (real-linear) operators $a\in \cL(F)$ with finite propagation. In order to show that $\broe^r(d,\cE)$ is a Banach sublattice of $\cL(E)_r$, we need to show that $\mathrm{B}_{u,\R}^r(d,\cE)$ is a real Banach sublattice of the real Banach lattice $\cL(F)_r$ and notice that $\broe^r(d,\cE)$ is the lattice complexification of $\mathrm{B}_{u,\R}^r(d,\cE)$.

The fact that  $\mathrm{B}_{u,\R}^r(d,\cE)$ is a real Banach sublattice of  $\cL(F)$   follows straightforwardly from the continuity of the lattice operations. Indeed, if $a,b\in \mathrm{B}_{u,\R}^r(d,\cR)$ and $(a_n)_n$ and $(b_n)_n$ are sequences of real band operators converging in the $r$-norm to $a$ and $b$, respectively, then $a\wedge b=\|\cdot\|_r\text{-}\lim_na_n\wedge b_n$ and $a\vee b=\|\cdot\|_r\text{-}\lim_na_n\vee b_n$. Since $a_n$ and $b_n$ are band operators, so are $a_n\wedge b_n$ and $a_n\vee b_n$, and we are done.

By the definition of the modulus, for any sequences $(a_n)_n$ and $(b_n)_n$ in $\cL(F)$ so that $\|\cdot\|_r\text{-}\lim_n(a_n+ib_n)=a+ib$ we must have that $\|\cdot\|_r\text{-}\lim_na_n=a$ and $\|\cdot\|_r\text{-}\lim_na_n=a$. Therefore, if $a\in \broe^r(d,\cE)$, then    $a_r,a_c\in \mathrm{B}_{u,\R}^r(d,\cE)$. So the standard isomorphism between $\cL(F_\C)$ and $\cL(F)_\C$ restricts to an isomorphism between $\broe^r(d,\cE)$ and  $\mathrm{B}_{u,\R}^r(d,\cE)_\C$, and we are done.
\end{proof}

\section{Rigidity for order preserving equivalences}\label{SectionRigidityOrder}
In this section,   we investigate what kind of order preserving equivalences between two given uniform Roe algebras $\broe(d,\cE)$ and $\broe(\partial,\cF)$ are strong enough in order to guarantee that the base metric spaces are coarsely equivalent to each other.  

We start this section showing that an order preserving Banach space isomorphism is not enough for rigidity to hold (Theorem \ref{ThmBanachIsomorNorEnough}). Then,   we show that order preserving Banach space \emph{isometries} and  order preserving Banach \emph{algebra} isomorphisms do give us rigidity under reasonable  conditions on $\cE$ and $\cF$  (Theorem \ref{ThmBanachIsometryOrderIso} and Theorem \ref{ThmBanachAlgIsomOrderIso}).

\subsection{Banach space isomorphism and nonrigidity}
In this subsection, we show that order preserving  Banach space isomorphism is too weak of a property to give us rigidity.  For that, it will  be useful to  work with metric spaces admitting infinite-valued metrics. Precisely, if $(X,d)$ is an infinite-valued metric space (i.e., $d:X\times X\to [0,\infty]$ satisfies all the metric axioms but it may have infinite values), then we define $\cstu(X)$ in the exactly same way as for standard metric spaces.\footnote{The reader familiar with coarse spaces can also see an infinite-valued metric space $(X,d)$ as a coarse space, and $\cstu(X)$ as the uniform Roe algebra of this coarse space. We refer to \cite{RoeBook} for details on coarse spaces and \cite{BragaFarah2018} for the uniform Roe algebra of arbitrary coarse spaces}  

Given metric spaces $(X_1,d_2)$ and $(X_2,d_2)$, we write $X=X_1\sqcup X_2$ to denote the coarse disjoint union of $X_1$ and $X_2$, i.e., $X$ is the infinite-valued metric space endowed with the metric $d$ defined by \[d(x,y)=\left\{\begin{array}{ll}
d_1(x,y),     &\text{ if }(x,y)\in X_1\times X_1  \\
d_2(x,y),     & \text{ if }(x,y)\in X_2\times X_2\\
\infty,    & \text{ otherwise. }
\end{array}\right.\]
If $(X_i)_{i=1}^k$ are metric spaces, the sum $\bigsqcup_{i=1}^kX_i$ is defined analogously.

\begin{proposition}\label{PropN4AndN}
There exists an ordered Banach space isomorphism between $\cstu(\N)$ and $\bigoplus_{i=1}^4\cstu(\N) $. In particular, $\cstu(\N)$ and $\cstu(\bigsqcup_{i=1}^4\N)$ are isomorphic as ordered Banach spaces, but $\N$ and $\bigsqcup_{i=1}^4\N$ are not coarsely equivalent.
\end{proposition}

\begin{proof}
Let $I$ and $P$ denote the subsets of odd and even natural numbers, respectively.  Let $U_I:\ell_2(\N)\to \ell_2(I)$ and $U_P:\ell_2(\N)\to \ell_2(P)$ be the canonical isometries, i.e., $U_I\delta_n=\delta_{2n-1}$ and $U_P\delta_n=\delta_{2n}$ for all $n\in\N$.     It is straightforward that
\[a\in \cstu(\N)\mapsto (U_I^{-1}\chi_Ia U_I,U_P^{-1}\chi_P a U_P,U_I^{-1}\chi_I a U_P,U_P^{-1}\chi_Pa U_I)\in \bigoplus_{i=1}^4\cstu(\N)\]
is a Banach space isomorphism. Indeed, this can be seen by looking at the maps $\cstu(\N)\to \cstu(\N)$ induced by the bijections
\begin{align*}
    (n,m)\in \N\times\N &\mapsto(2n,2m)\in P\times P,\\ (n,m)\in \N\times\N &\mapsto(2n-1,2m-1)\in I\times I,\\ (n,m)\in \N\times\N &\mapsto(2n,2m-1)\in P\times I,  
    \text{ and}\\
    (n,m)\in \N\times\N &\mapsto(2n-1,2m)\in I\times P.
\end{align*}
 
Since $\bigoplus_{i=1}^4\cstu(\N)$ is canonically isomorphic (even as a \cstar-algebra) to $\cstu(\bigsqcup_{i=1}^4\N)$, the result follows.
\end{proof}

We now show that the (finite-valued) metric space version of Proposition \ref{PropN4AndN} if also true. Precisely, we have the following:

\begin{theorem}\label{ThmBanachIsomorNorEnough}
Let $X=(\{0\}\times \Z)\cup(\Z\times \{0\})\cup\{(n,\pm n)\mid n\in\Z\}$. Then $\N$ and $X$ are not coarsely equivalent, but there  exists an ordered Banach space isomorphism between  $\cstu(\Z)$ and $\cstu(X)$. In particular, $\mathrm{C}^{*,r}_u(\N)$ and $\mathrm{C}^{*,r}_u(X)$ are isomorphic as Banach lattices.
\end{theorem}

\begin{proof}
Throughout this proof, we use the word ``isomorphism'' meaning a Banach space isomorphism which is also an order isomorphism. 

Firstly, notice that $\cstu(\Z)$ is isomorphic to $\cstu(\N)^2\oplus \cK(\ell_2(\N))^2$. Indeed, let $A=\N$ and $B=\{n\in \Z\mid n\leq 0\}$. Then
\[a\in \cstu(\Z)\mapsto (\chi_Aa\chi_A,\chi_Ba\chi_B, \chi_Ba\chi_A,\chi_Aa\chi_B)\in \bigoplus_{i=1}^4\cstu(\Z)\]
defines an isomorphic embedding onto the sum of $\chi_A\cstu(\Z)\chi_A$, $\chi_B\cstu(\Z)\chi_B$, $\chi_B\cstu(\Z)\chi_A$ and $\chi_A\cstu(\Z)\chi_B$. Since $\chi_A\cstu(\Z)\chi_A$ and  $\chi_B\cstu(\Z)\chi_B$ are isomorphic to $\cstu(\N)$, and $\chi_B\cstu(\Z)\chi_A$ and  $\chi_A\cstu(\Z)\chi_B$ are isomorphic to $\cK(\ell_2(\N))$, we conclude that $\cstu(\Z)$ is isomorphic to $\cstu(\N)^2\oplus \cK(\ell_2(\N))^2$.

Similarly, $\cstu(X)$ is isomorphic to $\cstu(\N)^8\oplus \cK(\ell_2(\N))^{56}$. Indeed, write $X=\bigcup_{i=1}^4\Z_i$, where each $\Z_i$ contains $(0,0)$ and it is coarsely equivalent to $\Z$ by a map sending $(0,0)$ to $0$.  Under this coarse equivalence, it makes sense to talk about the positive and negative elements of each $\Z_i$. For each $i$, let $A_i$ be the positive elements of $\Z_i$ and  $B_i$ be the nonpositive elements of   $\Z_i$. Then the number 8 in $\cstu(\N)^8\oplus \cK(\ell_2(\N))^{56}$ comes from the corners $\chi_{A_i}\cstu(X)\chi_{A_i}$ and $\chi_{B_i}\cstu(X)\chi_{B_i}$, and the number 56 from all the possible combinations of $\chi_{C}\cstu(X)\chi_{D}$ where $C$ and $D$ are distint elements of $\{A_i,B_i\}_{i=1}^4$. 

By Proposition \ref{PropN4AndN}, $\cstu(\N)^2$ is isomorphic to $\cstu(\N)^8$. Hence, we only need to notice that $\cK(\ell_2(\N))$ is isomorphic to $\cK(\ell_2(\N))^k$ for any $k\in\N$. For that, given $k\in\N$, let $I_i=\{kn-i\mid n\in\N\}$ for each $i\in\{0,\ldots, k-1\}$. Then notice that 
\[a\in \cK(\ell_2(\N))\mapsto (a\chi_{I_i})_{i=0}^k\in \bigoplus_{i=0}^k \cK(\ell_2(I_i),\ell_2(\N))\]
is an isomorphism and that each $\cK(\ell_2(I_i),\ell_2(\N))$ is isomorphic to $\cK(\ell_2(\N))$. This proves the first assertion.

For the last statement, notice that  isomorphism $\Phi:\cstu(\Z)\to \cstu(X)$ constructed above restricts to a positive linear map $\Psi:(\mathrm{C}^{*,r}_u(\Z),\|\cdot\|_r)\to (\mathrm{C}^{*,r}_u(X),\|\cdot\|_r)$ whose inverse is also positive. Therefore, by Proposition \ref{PropPositiveAreBounded}, these restrictions are bounded. Hence $\Psi:(\mathrm{C}^{*,r}_u(\Z),\|\cdot\|_r)\to (\mathrm{C}^{*,r}_u(X),\|\cdot\|_r)$ is both a Banach space isomorphism and an order isomorphism. By \cite[Exercise 1.3.E2]{Meyer-NiebergBook1991}, $\Psi$ is a Banach lattice isomorphism.
\end{proof}

\subsection{Rigidity for lattice isomorphisms}  In this subsection, we show that if $\Phi:\broe(d,\cE)\to \broe(\partial,\cF)$ is either a  \emph{isometric} order isomorphism or a Banach \emph{algebra} lattice isomorphism, then the base spaces are often bijectively coarsely equivalent --- the word ``often'' regards conditions on the   basis of $\cE$ and $\cF$, see Theorem \ref{ThmBanachIsometryOrderIso} and Theorem \ref{ThmBanachAlgIsomOrderIso}.  All the results in this section hold with no extra   geometric assumption on the metric spaces.

\begin{proposition}\label{PropLatticeIsoRankPres}
Let $(\N,d)$ and $(\N,\partial)$ be metric spaces, and let $E$ and $F$ be Banach spaces with  normalized bimonotone bases  $\cE$ and $\cF$, respectively. Let $\Phi:\broe(d,\cE)\to \broe (\partial, \cF)$ be an ordered Banach space isomorphism. Then for all $n,m\in\N$ there are $\ell,k\in\N$ and $\lambda \in[\|\Phi^{-1}\|^{-1}, \|\Phi\|]$ so that $\Phi(e_{n,m})=\lambda e_{\ell,k}$.
\end{proposition}

\begin{proof}
Fix $n,m\in\N$ and let   $a=e_{n,m}$.  Notice that if $0\leq c\leq a$, then  $c$ is a scalar multiple of $a$. Hence, as both $\Phi$ and $\Phi^{-1}$ are positive, this implies that the positive operator $b=\Phi(a)$ satisfies the same property, i.e., if $0\leq c\leq b$, then  $c$ is a scalar multiple of $b$. Hence, $b$ must have the required property.
\end{proof}

The notion of \emph{coarse-like maps} between uniform Roe algebras was introduced in \cite{BragaFarah2018,BragaFarahVignati2018}, and it has been an important tool in the study of rigidity. The next is a stronger version of it. The definition of coarse-like maps is  given in Definition \ref{DefiCoarseLike}.

\begin{definition}\label{DefiStrongCoarseLike}
Let $(\N,d)$ and $(\N,\partial)$ be  u.l.f. metric spaces, and $E$ and $F$ be Banach spaces with bases  $\cE$ and $\cF$, respectively.  A map $\Phi:\broe(d,\cE)\to \broe(\partial,\cF)$ is \emph{strongly coarse-like} if for all   $r>0$, there exists $s>0$ so that  $\propg(\Phi(a))<s$   for all $a\in \broe(d,\N)$ with   $\propg(a)<r$.
  \end{definition}

\begin{proposition}\label{PropStronglyCoarseLike}
Let $(\N,d)$ and $(\N,\partial)$ be metric spaces,   $E$ and $F$ be Banach spaces with unconditional bases $\cE$ and $\cF$, respectively, and let $\Phi:\broe(d,\cE)\to \broe (\partial, \cF)$ be an ordered Banach space isomorphism. If $\cE$ is $d$-subsymmetric, then $\Phi$ is strongly coarse-like.
\end{proposition}

\begin{proof}
Without loss of generality, assume  that $\cE$ and $\cF$ are normalized and bimonotone. Let $r>0$ and write $A=\{(n,m)\in \N\times \N\mid d(n,m)\leq r\}$. Since $\cE$ is $d$-symmetric, Proposition \ref{PropSubsymmetric} implies that $a=\sum_{(n,m)\in A}e_{n,m}$ is a well defined operator in $\broe(d,\cE)$. Hence  Proposition \ref{PropLatticeIsoRankPres} implies that for each $(n,m)\in A$, there exists $\lambda_{n,m}\in [\|\Phi^{-1}\|^{-1},\|\Phi\|]$ and $\sigma(n,m)\in \N\times\N$ so that $\Phi(e_{n,m})=\lambda_{n,m}e_{\sigma(n,m)}$.

\begin{claim}
$\supp(a)=\bigcup_{(n,m)\in A}\{\sigma(n,m)\}$.
\end{claim} 
\begin{proof}
Since $\lambda_{n,m}e_{\sigma(n,m)}\leq \Phi(a)$ for all $(n,m)\in A$, it is clear that \[\bigcup_{(n,m)\in A}\{\sigma(n,m)\}\subset \supp(\Phi(a)).\] For the other inclusion, let $(k,\ell)\in \supp(\Phi(a))$ and pick $\lambda>0$ so that $\lambda e_{k,\ell}\leq \Phi(a)$. Then $\Phi^{-1}(\lambda e_{k,\ell})\leq a$. By Proposition \ref{PropLatticeIsoRankPres}, $\Phi^{-1}(\lambda e_{k,\ell})$ must equal $\alpha e_{n,m}$ for some $(n,m)\in \N\times \N$ and $\sigma(n,m)=(k,\ell)$. 
\end{proof}

\begin{claim}
Let $(n,m)\in A$ and $(k,\ell)=\sigma(n,m)$. Then $e^*_\ell(\Phi(a)e_k)=\lambda_{n,m}$.
\end{claim}

\begin{proof}
Since $\lambda_{n,m}e_{\sigma(n,m)}\leq \Phi(a)$, we must have that $e^*_\ell(\Phi(a)e_k)\geq\lambda_{n,m}$. Let $\lambda=e^*_\ell(\Phi(a)e_k)-\lambda_{n,m}$. Then  $e_{n,m}+\lambda\Phi^{-1}(e_{k,\ell})\leq a$. So $\lambda=0$.
\end{proof}

Since $\lambda_{n,m}>\|\Phi^{-1}\|^{-1}>0$ for all $(n,m)\in \N\times \N$, the previous claims imply that  \[s=\sup\{\partial(k,\ell) \mid (k,\ell)\in \supp(\Phi(a))\}<\infty.\] 
This implies that $\propg(\Phi(b))<s$   for all positive $b\in \broe(d,\N)$ with   $\propg(b)\leq r$. By Proposition \ref{PropFinPropImpliesRegular} (see Appendix), every $b\in \broe[d,\cE]$ can be written as a finite linear combination of positive  operators whose support are contained in the one of $b$, so the general result follows.  
\end{proof}

Notice that every separable Banach space can be (easily) renormed to be strictly convex (see  Section \ref{SectionIntro} for the definition of strict convexity and   \cite[Page 33]{JohnsonLindenstrauss2001Handbook} for this renorming result). If $\cE=(e_n)_n$ is a normalized monotone basis of a strictly convex space $E$, then $\|e_n+e_m\|>1$ for all $n,m\in\N$. Hence, the operator $e_{n,\ell}+e_{n,m}\in \cL(E)$ has norm greater than 1. Also, given $n\neq m$, there exists $\lambda>1/2$ so that $\|\lambda(e_n+e_m)\|\leq 1$.  Hence, the operator $e_{n,n}+e_{m,n}\in \cL(E)$ has norm greater than 1. Those   observations will be used in the following proof.

\begin{proof}[Proof of Theorem \ref{ThmBanachIsometryOrderIso}]
Without loss of generality, assume $\cE$ and $\cF$ are normalized. Given $n,m\in X$, let  $\sigma(n,m)\in Y\times Y$ be given by Proposition \ref{PropLatticeIsoRankPres}, i.e., $\Phi(e_{n,m})= e_{\sigma(n,m)}$ for all $n,m\in X$ --- here we use that $\Phi$ is an isometry, so $\lambda=1$. Clearly, $\sigma$ is a bijection and $\Phi^{-1}(e_{k,\ell})=e_{\sigma^{-1}(k,\ell)}$ for all $k,\ell\in Y$.  Let $\pi_1,\pi_2:Y\times Y\to Y\times Y$ denote the projections onto the first and second coordinates, respectively.

\begin{claim}
There exists $i\in \{1,2\}$ so that  $\pi_i(\sigma(n,n))=\pi_i(\sigma(n,m))$ for all $n,m\in X$.
\end{claim}

\begin{proof}
First notice that the claim holds for a fixed $n\in\N$. Indeed, fix $n\in\N$. Given $m\neq n$,  the comments preceding this theorem say that $e_{n,n}+e_{n,m}$ has norm greater than 1, hence so does  $\Phi(e_{n,n}+e_{n,m})=e_{\sigma(n,n)}+e_{\sigma(n,m)}$. As $\cF$ is $1$-symmetric, either  $\pi_1(\sigma(n,n))=\pi_1(\sigma(n,m))$ or $\pi_2(\sigma(n,n))=\pi_2(\sigma(n,m))$, and only one of those options can happen. Assume the former (if the latter holds, the proof proceeds analogously). If $\ell\neq m$, the same arguments give that either  $\pi_1(\sigma(n,n))=\pi_1(\sigma(n,\ell))$ or $\pi_2(\sigma(n,n))=\pi_2(\sigma(n,\ell))$. Hence, if $\pi_1(\sigma(n,n))\neq \pi_1(\sigma(n,\ell))$, then $e_{\sigma(n,m)}+e_{\sigma(n,\ell)}$ has norm 1; contradiction. 

This shows that for each $n\in X$, there exists $i_n\in \{1,2\}$ so that \[\pi_{i(n)}(\sigma(n,n))=\pi_{i(n)}(\sigma(n,m))\] for all $m\in X$. In order to notice that $i_n=i_m$ for all $n,m\in X$, one only needs to notice that the same proof holds for the map $\tau(n,m)=\sigma(m,n)$. 
\end{proof}

To simplify notation, assume the previous claim is satisfied with $i=1$. Define a map $f:X\to Y$ by letting $f(n)=\pi_1(\sigma(n,n))$ for all $n\in\N$. Since $\sigma$ is a bijection, the previous claim implies that $f$ is surjective. If $n\neq m$, then $\|e_{n,n}+e_{m,m}\|=1$. Hence $e_{\sigma(n,n)}+e_{\sigma(m,m)}$ also has norm 1, which implies that  $\pi_{1}(\sigma(n,n))\neq\pi_{1}(\sigma(m,m)) $. So, $f$ is also injective.

\begin{claim}
 $f$ and $f^{-1}$ are  coarse.
\end{claim}

\begin{proof}
Let $r>0$. Since $\Phi$ is strongly coarse-like (Proposition \ref{PropStronglyCoarseLike}), there exists $s>0$ so that $\partial(\sigma(n,m))\leq s$ for all $n,m\in X$ with $d(n,m)\leq r$. Fix distinct $n,m\in X$ with $d(n,m)\leq r$. By the comments preceding this theorem,   both $e_{n,n}+e_{n,m}$ or $e_{m,m}+e_{n,m}$ must have norm greater than $1$. Hence, as $\Phi$ is an isometry, both $e_{\sigma(n,n)}+e_{\sigma(n,m)}$ and $e_{\sigma(m,m)}+e_{\sigma(n,m)}$ must have norm greater than $1$ as well. Again by the isometric property of $\Phi$,  $e_{\sigma(n,n)}+e_{\sigma(m,m)}$ has norm 1. 

Therefore, as $\cF$ is $1$-symmetric, we must have that either $\pi_1(\sigma(n,n))=\pi_1(\sigma(n,m))$ and $\pi_2(\sigma(m,m))=\pi_2(\sigma(n,m))$ or $\pi_2(\sigma(n,n))=\pi_2(\sigma(n,m))$ and $\pi_1(\sigma(m,m))=\pi_1(\sigma(n,m))$. Without loss of generality, suppose the latter holds. Then 
\begin{align*}\partial(f(n),f(m))&=\partial(\pi_1(\sigma(n,n)),\pi_1(\sigma(m,m)))\\
&\leq \partial(\pi_1(\sigma(n,n)),\pi_2(\sigma(n,m)))+\partial(\pi_2(\sigma(n,m)),\pi_1(\sigma(n,m)))\\
&\leq \partial(\pi_1(\sigma(n,n)),\pi_2(\sigma(n,n)))+s\\
&\leq 2s.
\end{align*}
This concludes that $f$ is coarse.  

Notice that $f^{-1}(k)=\pi_1(\sigma^{-1}(k,k))$. Hence $f^{-1}$ is coarse by the same arguments applied to $\sigma^{-1}$ and $\Phi^{-1}$.
\end{proof}
This concludes the proof. 
\end{proof}

\begin{theorem}\label{ThmBanachAlgIsomOrderIso} 
Let $(\N,d)$ and $(\N,\partial)$ be u.l.f. metric spaces, and $E$ and $F$ be Banach spaces with $d$-symmetric and $\partial$-symmetric bases $\cE$ and $\cF$, respectively.
If  $\broe(d,\cE)$ and $\broe(\partial,\cF)$ are simultaneously order and Banach algebra isomorphic, then $(\N,d)$ and $(\N,\partial)$ are bijectively coarsely equivalent. 
\end{theorem}

\begin{proof} 
 Since $\Phi$ is a homomorphism, $\Phi(e_{n,n})$ must be an idempotent for all $n\in\N$.   So Proposition \ref{PropLatticeIsoRankPres} gives a map $f:\N\to \N$ so that $\Phi(e_{n,n})=e_{f(n),f(n)}$ for all $n\in \N$.  Clearly, $f$ is a bijection. We are left to notice that $f$ is a coarse equivalence. For that, it is enough to show that both $f$ and $f^{-1}$ are coarse. Fix $r>0$ and let $s>0$ be so that $\partial(k,\ell)\leq s$ for all $n,m\in\N$ with $d(n,m)\leq r$ and $(k,\ell)=\supp(\Phi(e_{n,m}))$  (Proposition \ref{PropStronglyCoarseLike}). Since $\Phi(e_{m,m})\Phi(e_{n,m})=\Phi(e_{n,m})\Phi(e_{n,n})$, we must have that $\partial (f(n),f(m))\leq s$ for all $n,m\in\N$ with $d(n,m)\leq r$. So $f$ is coarse. Analogous arguments show that $f^{-1}$ is coarse, so we are done.
\end{proof}

\begin{proof}[Proof of Theorem \ref{ThmBanachAlgIsomOrderIsoINTRO}]
As symmetric basis are  $d$-symmetric, this  is a particular case of  Theorem \ref{ThmBanachAlgIsomOrderIso}.
\end{proof}

\section{Property A and ghost operators}\label{SectionPropA}

Section \ref{SectionRigidityOrder} dealt with rigidity for order preserving isomorphisms. For now on, we forget order preservation, and our goal is to obtain rigidity for some uniform Roe algebras under Banach algebra isomorphisms. For that, extra geometric conditions on our metric spaces will be needed. Precisely, we will need to assume that the ghost operators in $\broe^r(d,\cE)$ belong to $\mathrm M_\infty(\cE)$. Since this is a technical condition, in order to obtain examples, we start by showing in this section that  metric spaces with property A satisfy this technical condition (see Theorem \ref{ThmPropertyAGhostsAreComp}).

\begin{definition}\label{DefiMetricPartition}
Let $d$ be a metric on $\N$ and $p\in (1,\infty)$. A sequence $(\varphi_n:\N\to [0,1])_{n\in \N}$ is called a \emph{$d$-partition of unit} if the following holds.
\begin{enumerate}
    \item $\sup_{m\in\N}|\{n\in \N\mid \varphi_n(m)\neq 0\}|<\infty$,
    \item $\sup_{n\in\N}\diam(\supp(\varphi_n))<\infty$, and 
    \item $\sum_{n\in\N}\varphi_n(m)^2=1$ for all $m\in\N$.
\end{enumerate}
Moreover, if $r,\eps>0$, we say that the sequence  $(\varphi_n:\N\to [0,\infty))_{n\in \N}$ has \emph{$(r,\eps)$-variation} if $d(k,m)<r$ implies
\[ \sum_{n\in\N}|\varphi_n(k)- \varphi_n(m)|^2<\eps^2.\]
\end{definition}

\begin{definition}\label{DefiPropA}
A u.l.f.  metric space $(\N,d)$ has \emph{property A} if for all $r,\eps>0$ there exists a $d$-partition of unit with $(r,\eps)$-variation.
\end{definition}

The definition above is not the original definition of property A given by G. Yu in \cite{Yu2000}, but it is equivalent to it by \cite[Theorem 1.2.4]{Willett2009}.\footnote{Notice also that the number $2$ could be replaced by any $p\in (1,\infty)$ in the definition of $d$-partition of unit and of property A \cite[Theorem 1.2.4]{Willett2009}.}

The next result is the version of \cite[Lemma 6.3]{SpakulaWillett2017} for our Banach lattice setting.

\begin{lemma}\label{LemmaSeriesSOTconv}
Let $(\N,d)$ be a u.l.f.  metric space,  and $(\varphi_n:\N\to [0,1])_{n\in \N}$ be a $d$-partition of unit. Let $E$ be a Banach space with a $1$-unconditional basis $\cE$. Then 
\begin{align*}
M:(\cL(E)_r,\|\cdot\|_r)&\to (\broe[d,\cE],\|\cdot\|_r)\\
b&\mapsto \mathrm{SOT}\text{-}\sum_{n\in \N}\varphi_nb\varphi_n
\end{align*}
is a well defined positive operator with norm at most $4$.
\end{lemma}

\begin{proof}
Fix $b\in \cL(E)_r$. By the definition of a $d$-partition of unit, it follows that   $\{n\in\N\mid \varphi_n\xi\neq 0\}$ is finite  if $\xi\in E$ has finite support.  Hence, in order to show that $\sum_{n\in\N}\varphi_nb\varphi_n\xi$  converges strongly for any $\xi\in E$, and that $\sum_{n\in \N}\varphi_nb\varphi_n$ is an operator  with $r$-norm at most $4\|b\|_r$, it is enough to show that 
\[\Big\|\sum_{n\in \N}\varphi_nc\varphi_n\xi\Big\|\leq \|b\|_r\|\xi\|\]
for all $\xi\in E$ with finite support, and all $c\in\{b_{1},b_{2},b_{3}, b_{4}\}$, where $b=b_{1}-b_2+i(b_3-b_4)$ is the canonical decomposition of $b$ into positive elements.   For that, we  fix  $\xi\in E$ with finite support, $c\in\{b_{1},b_{2},b_{3}, b_{4}\}$,  an arbitrary  $\zeta\in E^*$ with finite support, and show that 
\[\Big|\zeta\Big(\sum_{n\in\N}\varphi_nc\varphi_n\xi\Big)\Big|\leq \|b\|_r\|\xi\|\|\zeta\|.\]

First notice that the adjoint of the diagonal operator  $\varphi_n \in \cL(E)$ is the diagonal operator $\varphi_n\in \cL(E^*)$. So,  
\begin{align*}
\Big|\zeta\Big(\sum_{n\in\N}\varphi_nc\varphi_n\xi\Big)\Big|=\Big|\sum_{n\in \N}\big(\varphi_n \zeta \big)(c\varphi_n\xi )\Big|.
\end{align*}
Clearly,  each $\varphi_n\zeta$ is regular. Hence,   Proposition \ref{PropDedekindBound} and Proposition \ref{PropHolderForLattices} give that  
\[ \Big|\sum_{n\in\N}\big(\varphi_n \zeta \big)(c\varphi_n\xi ) \Big|\leq\sum_{n\in\N}\big|\varphi_n \zeta \big||c\varphi_n\xi  | \leq \Big(\sum_{n\in\N}|\varphi_n  \zeta |^2\Big)^{1/2}\Big(\sum_{n\in\N}|c\varphi_n\xi |^2\Big)^{1/2}.\]
(see Appendix for the definition of $(\sum_{i}^n|x_i|^2)^{1/2}$).

Since $c$  is positive, Proposition \ref{PropDedekindBound} implies that $|c\varphi_n\xi |\leq |c||\varphi_n\xi |$. As $c\leq |b|$, we have that $\|c\|\leq \|b\|_r$. Hence,  using positivity of $c$ once again,  we have that  
\begin{equation}\label{EqBla1}
\Big\|\Big(\sum_{n\in\N}|c\varphi_n\xi |^2\Big)^{1/2}\Big\|\leq \|b\|_r\Big\|\Big(\sum_{n\in\N}|\varphi_n\xi |^2\Big)^{1/2}\Big\|
\end{equation}
(see \cite[Propostion 1.d.9]{LindenstraussTzafririVol2}).

\begin{claim}
We have that \[\Big(\sum_{n\in\N}|\varphi_n\xi |^2\Big)^{1/2}= |\xi|\ \text{ and }\ \Big(\sum_{n\in\N}|\varphi_n \zeta |^2\Big)^{1/2}=|\zeta| .\]
\end{claim}
\begin{proof}
Without loss of generality, assume all coordinates of $\xi$ are real. Let us show that  $(\sum_{n\in\N}|\varphi_n\xi |^2)^{1/2}= |\xi|$. First notice that
\begin{align*}
\Big(\sum_{n\in\N}|\varphi_n  \xi |^2\Big)^{1/2}&= \Big(\sum_{n\in\N}\Big|\sum_{k\in\N}\varphi_n(k)\xi(k)e_k\Big|^2\Big)^{1/2}\\
&=\bigvee\Big\{\sum_{n\in\N}\beta_n\sum_{k\in\N}\varphi_n(k)\xi(k)e_k\mid \sum_{n}\beta^2_n=1\Big\}\\
&=\bigvee\Big\{\sum_{k\in\N}\Big(\sum_{n\in\N}\beta_n\varphi_n(k)\Big)\xi(k)e_k\mid \sum_{n}\beta^2_n=1\Big\}.
    \end{align*}
Given $(\beta_n)_n$ with $\sum_n\beta_n^2=1$, the classic H\"older's inequality gives that \[\Big|\sum_{n\in\N}\beta_n\varphi_n(k)\Big|\leq \Big(\sum_n\beta_n^2\Big)^{1/2}\Big(\sum_{n}\varphi_n(k)^2\Big)^{1/2}=1\] 
for all $k\in\N$.  This implies that
\[\Big(\sum_{n\in\N}|\varphi_n  \xi |^2\Big)^{1/2}\leq \sum_{k\in\N}|\xi(k)|e_k=|\xi|\]
In order to get equality, fix $k\in\N$, let $\eps=\xi(k)/|\xi(k)|$ (let $\eps=0$ if $\xi(k)=0$), and let $(\beta_n)_n=(\eps\varphi_n(k))_n$. This  implies that  $|\xi(k)|e_k\leq(\sum_{n\in\N}|\varphi_n\xi |^2)^{1/2}$ for each $k\in\N$. So, $|\xi|\leq (\sum_{n\in\N}|\varphi_n\xi |^2)^{1/2}$.  

The equality $(\sum_{n\in\N}|\varphi_n\zeta |^2)^{1/2}= |\zeta|$ follows completely analogously. 
\end{proof}

The previous claim and  \eqref{EqBla1}  imply that
\[\Big|\zeta\Big(\sum_{n\in\N}\varphi_nc\varphi_n\xi\Big)\Big|\leq \|b\|_r\big\||\xi|\big\|\big\||\zeta|\big\|=\|b\|_r\|\xi\|\|\zeta\|.\]
This completes the proof that $\sum_{n\in \N}\varphi_nb\varphi_n$ defines an operator on $\cL(E)$ with $r$-norm at most $4\|b\|_r$. The definition of $d$-partition of unit clearly implies that this operator has finite propagation and it is positive. 
\end{proof}

Given a metric space $(\N,d)$ and a Banach space $E$ with basis $\cE$, we say that  an operator $v\in \cL(E)$ is a \emph{partial translation} if there exists a bijection $f:A\subset \N\to B\subset \N$ so that (1) $\sup_nd(n,f(n))<\infty$ and (2)  $ve_n=e_{f(n)}$ for all $n\in\ A$, and $ve_n=0$ for all $n\not\in A$. Notice that if $\cE$ is $d$-symmetric, then $\broe(d,\cE)$ contains all partial translations in $\cL(E)$.  The next simple result can be obtained by Lemma \ref{LemmaULFPartition} (or \cite[Lemma 2.4]{SpakulaWillett2017}) and we omit the details.

\begin{lemma}\label{LemmaDecompositionFinPropOperator}
Let $(\N,d)$ be a u.l.f  metric space, $E$ be a Banach space with normalized symmetric   basis $\cE$ and let $a\in \cL(E)$ be an operator with propagation at most $r$. There is $N\in\N$,   $f_1,\ldots, f_N\in \ell_\infty(\cE)$ with $\|f_i\|\leq \|a\|$ and partial translations $v_1,\ldots, v_N\in \broe(d,\cE)$ so that $a=\sum_{i=1}^Nf_iv_i$. 
\end{lemma}

The next result is the version of \cite[Lemma 6.4]{SpakulaWillett2017} for our Banach lattice setting. 

\begin{lemma}\label{LemmaSeriesSOTconvCommutator}
Let $(\N,d)$ be a u.l.f.  metric space  and $(\varphi_n:\N\to [0,1])_{n\in \N}$ be a $d$-partition of unit with $(r,\eps)$-variation, for some $r,\eps>0$. Let $E$ be a Banach space with a normalized $1$-unconditional symmetric basis $\cE$   and $a\in \cL(E)$   with $\propg(a)\leq r$. Then 
\[\sum_{n\in \N}\varphi_n[\varphi_n,a]\]
converges in the strong topology to a finite propagation operator with operator norm is at most $ \eps N\|a\| $, where $N=\sup_{n\in\N}|B_r(n)|$.
\end{lemma}

\begin{proof}
Following the approach of Lemma \ref{LemmaSeriesSOTconv}, we only need to show that 
\[\Big|\zeta\Big(\sum_{n\in\N}\varphi_n [\varphi_n,a]\xi\Big)\Big|\leq \eps N\|a\| \|\xi\|\|\zeta\|.\]
for any $\xi\in E$ and $\zeta\in E^*$ with finite support. Proceeding as in Lemma \ref{LemmaSeriesSOTconv}, Proposition \ref{PropHolderForLattices} gives us that   
\[\Big|\zeta\Big(\sum_{n\in\N}\varphi_n [\varphi_n,a]\xi\Big)\Big|\leq \Big(\sum_{n\in\N}|\varphi_n  \zeta |^2\Big)^{1/2}\Big(\sum_{n\in\N}|[\varphi_n,a]\xi |^2\Big)^{1/2},\]
and analogous computations as in  Lemma \ref{LemmaSeriesSOTconv} imply that
\[\Big\|\Big(\sum_{n\in\N}|\varphi_n \zeta |^2\Big)^{1/2}\Big\|\leq  \|\zeta\|.\]

We now show that  $(\sum_{n\in\N}|[\varphi_n,a]\xi |^2)^{1/2}\leq \eps N\|a\||\xi|$. As  $\propg(a)\leq r$, write $a=\sum_{i=1}^Nf_iv_i$, where each $f_i$ is a   diagonal operator in $\ell_\infty(\cE)$ with norm at most $\|a\|$ and each $v_i$ is a partial translation with propagation at most $r$ (see Lemma \ref{LemmaDecompositionFinPropOperator}). Using that diagonal operators commute with each other, we have that 
 \[|[\varphi_n,a]\xi |=\Big|\sum_{i=1}^Nf_i[\varphi_n,v_i]\xi \Big|\leq \|a\|\sum_{i=1}^N|[\varphi_n,v_i]\xi |  \]
 for each $n\in\N$.  
 
Let $n\in\N$ and $i\in \{1,\ldots, N\}$. Since $v_i$ is a partial translation with propagation at most $r$, let $t_i:R_i\subset \N\to D_i\subset \N$ be a bijection so that $v_i(e_{t_i(k)})=e_{k}$ and $d(k,t_i(k))\leq r$ for all $k\in R_i$ and $v_i(e_m)=0$ for all $k\not\in D_i$. Then one can easily see that 
\[\Big([\varphi_n,v_i]\xi\Big)(k)=\left\{\begin{array}{ll}
(\varphi_n(k)-\varphi_n(t_i(k)))\xi(t_i(k)),&   \  \text{ if } \  k\in R_i  \\
0,&   \  \text{ if } \  k\not\in R_i.
\end{array}\right.\]
This gives us that  
\begin{align*}
\Big(&\sum_{n\in\N}|[\varphi_n,a]\xi |^2\Big)^{1/2}\leq \|a\|\Big(\sum_{n\in\N}\sum_{i=1}^N|[\varphi_n,v_i]\xi|^2\Big)^{1/2}\\
&=\|a\|\bigvee\Big\{   \sum_{n\in\N}\sum_{i=1}^N\beta_n\sum_{k\in R_i} (\varphi_n(k)-\varphi_n(t_i(k)))\xi(t_i(k))   \mid \sum_n\beta_n^2=1\Big\}\\
&=\|a\|\bigvee\Big\{\sum_{i=1}^N   \sum_{k\in R_i}\Big( \sum_{n\in\N}\beta_n(\varphi_n(k)-\varphi_n(t_i(k)))\Big)\xi(t_i(k))   \mid \sum_n\beta_n^2=1\Big\}.
\end{align*}
Hence, using classic H\"older's inequality just as in the proof of Lemma \ref{LemmaSeriesSOTconv} and the fact that $(\varphi_n)_n$ has $(r,\eps)$-variation, we obtain that 
\[\Big(\sum_{n\in\N}|[\varphi_n,a]\xi |^2\Big)^{1/2}\leq \eps \|a\|\sum_{i=1}^N\sum_{k\in R_i}|\xi(t_i(k))|e_k\leq \eps N\|a\||\xi|.\]
 
The estimates above gives us that $\sum_{n\in \N}\varphi_n [\varphi_n,a]$ is well defined and it has norm no greater than $\eps N\|a\|$. Since $(\varphi_n)_n$ has $(r,\eps)$-variation, it is clear that this operator has finite propagation. We leave the details to the reader.
\end{proof}

\begin{lemma}\label{LemmaApproxPosByFinProg}
Let $E$ be a Banach space with a $1$-unconditional symmetric basis $\cE$. Let $(\N,d)$ be a uniformly locally finite metric space with property A,   and for each $k\in\N$, let $(\varphi_{n,k})_n $ be a $d $-partition of unit with $(k,1/k)$-variation. For each $k\in\N$ let 
\begin{align*}
M_k:\broe^r(d,\cE) &\to \broe[d,\cE]\\
b&\mapsto\mathrm{SOT}\text{-}\sum_{n\in\N}\varphi_{n,k}b\varphi_{n,k}.
\end{align*}
Then  $b=\|\cdot\|\text{-}\lim_kM_k(b)$ for all $b\in \broe^r(d,\cE)$.
\end{lemma}
 
\begin{proof}
Without loss of generality,   assume $\cE$ to be normalized.  Fix $b\in \broe^r(d,\cE)$ and $\eps>0$. Pick $a\in \broe[d,\cE]$ so that $\|b-a\|_r<\eps$. 

Since $\broe[d,\cE]\subset \cL(E)_r$ (Proposition \ref{PropFinPropImpliesRegular}), $M_k(a)$ is well defined. Moreover, notice that
\begin{align*}
M_k(a)&=\sum_{n\in\N}\varphi_{n,k}a\varphi_{n,k}\\
&=\sum_{n\in\N}\varphi^{2}_{n,k}a+\sum_{n\in\N}\varphi_{n,k}[a,\varphi_{n,k}]\\
&=a+\sum_{n\in\N}\varphi_{n,k}[a,\varphi_{n,k}].
\end{align*}
By Lemma \ref{LemmaSeriesSOTconvCommutator}, the operator norm of
$\sum_{n\in\N}\varphi_{n,k}[a_i,\varphi_{n,k}]$ is bounded by $ N^2\|a\|/k$, where $N=\sup_{n\in\N}|B_{\propg(a)}(n)|$. So we can pick  $k_0\in\N$ large enough so that $\|a-M_k(a)\|<\eps$ for all $k>k_0$. Therefore, we obtain that
\begin{align*}
\|b-M_k(b)\|&\leq \|b-a\|+\|a-M_k(a)\|+\|M_k(a-b)\|\\
&\leq \eps+\eps+\|a-b\|_r\\
&\leq 3\eps
\end{align*}
for all $k>k_0$. Since $\eps>0$ was arbitrary, this finishes the proof.
\end{proof}

\begin{proof}[Proof of Theorem \ref{ThmPropertyAGhostsAreComp}]
Fix a  ghost operator $a\in \broe^r(d,\cE)$. Let $(M_k)_k$ be as in Lemma \ref{LemmaApproxPosByFinProg}. The formula of $M_k$ clearly implies that each  $M_k(a)$ is also a ghost and has  finite propagation.  Hence, by Proposition \ref{PropGhostWithFInProgAreComp}, $M_k(a)\in \mathrm M_\infty(\cE)$ for all $k\in\N$. As $a=\lim_kM_k(a)$, this shows that $a\in \mathrm M_\infty(\cE)$. In particular, $a$ is compact.
\end{proof}

\section{Rigidity for Banach algebra isomorphisms}\label{SectionRigBanAlgIso}

In this section, we prove rigidity results for Banach algebra isomorphisms. Unlike the rigidity results in  Section \ref{SectionRigidityOrder}, the results in this section only hold under some technical geometric conditions on the u.l.f. metric spaces --- on the bright side, no order preservation is needed.  We first present our results under the technical  condition that ghost idempotents belong to $\mathrm{M}_\infty(\cE)$. In Subsection \ref{SubsectionPropARig}, we apply Theorem \ref{ThmPropertyAGhostsAreComp} in order to obtain results for spaces with property A.

\subsection{Choosing the coarse equivalence}\label{SubsectionChoosingMaps}
We start by presenting the method of choosing the map which will be later shown to be a coarse equivalence between the base metric spaces of the uniform Roe algebras. This is the only step in our proof  that uses  the extra geometric condition of the space.

\begin{proposition}\label{PropSOTConvSeriesRegNew}
 Let $(\N,\partial)$ be a u.l.f. metric space and   $F$ be a Banach space with a bimonotone  basis  $\cF$. Let $(p_n)_n$ be a sequence of rank 1 operators so that $\mathrm{SOT}\text{-}\sum_{n\in\M}p_n\in \broe(\partial,\cF)$ for all $\M\subset \N$ and $\inf_n\|p_n\|>0$. Given an infinite $\M_0\subset \N$ and $\eps,\rho\in (0,1)$, there exist an infinite $\M\subset \M_0$ and a normalized block subsequence of $\cF$  so that  $\|p_n\eta_n\|\geq \rho\|p_n\|$ and 
     $\|p_m\eta_n\|< 2^{-m}\eps$ for all $m\neq n$ in $\N$.
\end{proposition}

\begin{proof}
For each $n\in\N$, pick a unit vector $\xi_n\in F$ and $f_n\in F^*$ so that  $\|f_n\|=\|p_n\|$  and $p_n=\xi_n\otimes f_n$. Since  $\delta=\inf_n\|p_n\|>0$,  $(f_n)_n$ is semi-normalized. As  $\sum_{n\in\N}p_n$ converges in the strong operator topology, the sequence $(f_n)_n$ must be weak$^*$-null. Since $p_n\in \broe(\partial, \cF)$, each $p_n$ can be approximated arbitrarily well by band operators. Hence, it follows that \[\lim_{m\to \infty}\chi_{[m,\infty]}f_n=0\]
for each $n\in\N$. The result now follows from the bimonotonicity of the basis. 
\end{proof}

The following can be compared with \cite[Lemma 5.1]{BragaFarahVignati2019}.

\begin{lemma}\label{LemmaPickDelta}
Let $(\N,d)$ and $(\N,\partial)$ be u.l.f. metric spaces, and $E$ and $F$ be Banach spaces with $1$-unconditional bases $\cE$ and $\cF$, respectively. Let $\Phi:\broe(d,\cE)\to \broe(\partial,\cF)$ be a strongly continuous  rank preserving Banach space embedding.  Then  either 
\[\delta=\inf_{n\in\N }\sup_{m,k\in\N}\|e_{m,m}\phi(e_{n,n})e_{k,k}\|>0\]
or there exists an infinite $\M_0\subset \N$ so that $\Phi(\sum_{n\in\M}e_{n,n})$ is a ghost which does not belong to  $\mathrm{M}_\infty(\cF)$ for all infinite $\M\subset \M_0$.
\end{lemma}

\begin{proof}
Suppose $\delta=0$ and let us show that the second condition must hold. Since $\delta=0$,   pick a sequence $(x_n)_n$ in $\N$ so that $\|e_{m,m}\phi(e_{x_n,x_n})e_{k,k}\|<2^{-n}$ for all $m,k\in\N$. Since $\Phi$ is injective, by going to a subsequence, we can assume that $(x_n)_n$ is a sequence of distinct natural numbers. 

\begin{claim}
There exists  an infinite $\M\subset \N$ so that if $p=\Phi(\sum_{n\in\M}e_{x_n,x_n})$, then for all $\rho\in (0,1)$ and all $m\in\N$, there exists a unit vector $\xi\in \overline{\spann}\{f_i\mid i> m\}$ so that $\|p\xi\|\geq \rho \|\Phi^{-1}\|^{-1}$.
\end{claim}

\begin{proof} Fix $\eps>0$. By Proposition \ref{PropSOTConvSeriesRegNew},   there exists an infinite $\M\subset \N$ and a sequence $(\xi_n)_n$ of unit vectors so that $\xi_n\in \overline{\spann}\{f_i\mid i> n\}$, $\|\Phi(e_{x_n,x_n})\xi_n\|\geq \rho \|\Phi^{-1}\|^{-1}$ and $\|\Phi(e_{x_i,x_i})\xi_n\|<\eps \rho 2^{-i}\|\Phi^{-1}\|^{-1}$ for all $n\in\N$ and all $i\neq n$. Hence, since  $p=\mathrm{SOT}\text{-}\sum_{n\in\M}\Phi(e_{x_n,x_n})$, we have that
\[\|p\xi_n\|\geq  \|\Phi(e_{x_n,x_n})\xi_n\|-\sum_{i\in \M\setminus\{n\}}\|\Phi(e_{x_i,x_i})\xi_n\|\geq \rho(1-\eps)\|\Phi^{-1}\|^{-1},\]
and the claim follows.
\end{proof}

Let $\M_0$ be given by the previous claim, fix an infinite $\M\subset \M_0$  and set  \[p=\Phi\Big(\sum_{n\in\M}e_{x_n,x_n}\Big).\]  For any $m\in\N$, $\rho\in (0,1)$ and $a\in \mathrm{M}_m(\cE)$, the previous claim gives   a unit vector $\xi\in \overline{\spann}\{f_i\mid i> m\}$ so that $\|p\xi\|\geq \rho \|\Phi^{-1}\|^{-1}$. Therefore,
\[\|(p-a)\xi\|=\|p\xi\|\geq   \rho\|\Phi^{-1}\|^{-1},\]
so $p$  is at least $\|\Phi^{-1}\|^{-1}$ apart from  $\mathrm{M}_\infty(\cF)$.

\begin{claim}
The operator $p$ is a ghost. 
\end{claim}

\begin{proof}
Let $\eps>0$. Pick $n_0\in\N$ so that $2^{n_0}\leq \eps/2$. Since each $\Phi(e_{x_n,x_n})$ is a ghost, fix a finite $A\subset \N$ so that $\|e_{m,m}\phi(e_{x_n,x_n})e_{k,k}\|<\eps/(2n_0)$ for all $m,k\not\in A$ and all $n\leq n_0$. Then, using again that  $p=\mathrm{SOT}\text{-}\sum_{n\in\M}\Phi(e_{x_n,x_n})$, we have that
\[\|e_{m,m}pe_{k,k}\|\leq \sum_{i=1}^{n_0}\|e_{m,m}\phi(e_{x_n,x_n})e_{k,k}\|+\sum_{i=n_0+1}^{\infty}\|e_{m,m}\phi(e_{x_n,x_n})e_{k,k}\|\leq \eps\]
for all $m,k\not\in A$, and the claim is proved.
\end{proof}
This concludes our proof.
\end{proof}

Let   $B\subset A$ be Banach algebras.
Inspired by \cstar-algebra theory, we say that   $B$ is a \emph{hereditary subalgebra of $A$} if  $BAB \subset B$. The following was proved in \cite[Lemma 3.11]{BragaVignati2019} for $\cE$ being the standard basis of $\ell_p$.  Since the exactly same proof holds for any basis,  we omit its proof here. 
 
\begin{lemma}\label{LemmaPhiIsStrongContAndRankPres}
Let $(\N,d)$ and $(\N,\partial)$ be u.l.f. metric spaces, and $E$ and $F$ be Banach spaces with bases $\cE$ and $\cF$, respectively. If $\Phi:\broe(d,\cE)\to \broe(\partial,\cF)$ is a Banach algebra embedding onto a hereditary Banach subalgebra of $\broe(\partial,\cF)$, then  there exists a surjective isomorphism $U: E\to \mathrm{Im}(\Phi(1))$ so that 
\[\Phi(a)=UaU^{-1}\Phi(1)\]
for all $a\in \broe(d,\cE)$. Moreover, if $\Phi$ is an isometry, so is $U$. In particular,  $\Phi$ is  strongly continuous and rank preserving.
\end{lemma}

The next corollary spells out our technical geometric condition: \emph{all   ghost idempotents in the uniform Roe algebra $\broe(\partial,\cE)$  belong to  $\mathrm M_\infty(\cE)$}.

\begin{corollary}\label{CorPickDelta2}
Let $(\N,d)$ and $(\N,\partial)$ be metric spaces and $E$ and $F$ be Banach spaces with $1$-unconditional bases $\cE$ and $\cF$, respectively.  Let $\Phi:\broe(d,\cE)\to \broe(\partial,\cF)$ be a Banach algebra embedding onto a hereditary subalgebra of $\broe(\partial,\cF)$.  If  all   ghost idempotents in $\broe(\partial,\cF)$ belong to $\mathrm M_\infty(\cF)$, then  
\[\delta=\inf_{n\in\N }\sup_{m,k\in\N}\|e_{m,m}\phi(e_{n,n})e_{k,k}\|>0.\]
\end{corollary}

\begin{proof}
This follows from Lemma \ref{LemmaPickDelta} and  Lemma \ref{LemmaPhiIsStrongContAndRankPres}.  
\end{proof}

\subsection{Coarse-like maps}\label{SubsectionCoarseLike} Coarse-like maps between uniform Roe algebras were introduced in \cite{BragaFarah2018,BragaFarahVignati2018} for $\ell_2$. The next definition generalizes it to our setting (cf. Definition \ref{DefiStrongCoarseLike}).

\begin{definition}\label{DefiCoarseLike}
Let $(\N,d)$ and $(\N,\partial)$ be u.l.f. metric spaces, and let $E$ and $F$ be Banach spaces with bases $\cE$ and $\cF$,  respectively. 
\begin{enumerate}
    \item Let $\eps,r>0$. An operator $a\in \broe(d, \cE)$ is \emph{$\eps$-$r$-approximable} if there exists $b\in \broe(X,\cE)$ with $\propg(b)\leq r$ and $\|a-b\|\leq\eps$.
    \item A map $\Phi:\broe(d,\cE)\to \broe(\partial,\cF)$ is \emph{coarse-like} if for all $\eps>0$ and all $r>0$, there exists $s>0$ so that  $\Phi(a)$ is $\eps$-$s$-approximable for all $a\in \broe(d,\N)$ with $\|a\|\leq 1$ and $\propg(a)\leq r$.
\end{enumerate}
\end{definition}

The following follows straightforwardly from the definition of coarse-like maps (cf. \cite[Lemma 3.2]{BragaVignati2019}). 

\begin{lemma}\label{LemmaCloseMapsAndUnfFiniteToOne}
Let $(\N,d)$ and $(\N,\partial)$ be metric spaces, and $E$ and $F$ be Banach spaces with   bases $\cE$ and $\cF$, respectively. Let $\Phi:\broe(d,\cE)\to \broe(\partial,\cF)$ be a coarse-like map,   let $f,g:\N\to \N$ be functions and $\delta>0$. If \[\|e_{f(n),f(n)}\Phi(e_{n,n})e_{g(n),g(n)}\|\geq \delta\] for all $n\in\N$, then $f$ and $g$ are close.\qed
\end{lemma}

\begin{theorem}\label{ThmCoarseLike}
Let $(\N,d)$ and $(\N,\partial)$ be u.l.f. metric spaces, and let $E$ and $F$ be Banach spaces with unconditional bases $\cE$ and $\cF$, respectively. If   $\cE$ is $d$-symmetric, then every strongly continuous rank preserving linear map $\Phi:\broe(d,\cE)\to \broe(\partial, \cF)$ is coarse-like. 
\end{theorem}

Theorem \ref{ThmCoarseLike} is the  version of  \cite[Proposition 3.3]{BragaFarahVignati2019} to our settings. In order to prove it, we need the next lemma
 (c.f \cite[Lemma 4.9]{BragaFarah2018}).  We use the notation $\D=\{z\in \C\mid |z|\leq 1\}$.

\begin{lemma}\label{Lemma49}
Let $(\N,d)$   be a u.l.f. metric space, and let $E$  be a Banach space  with an unconditional basis $\cE$. Let $(a_n)_n$ be a sequence of finite rank operators in $\broe(d,\cE)$ so that $\sum_{n}\alpha_na_n$ converges in the strong operator topology to an operator in $\broe(d,\cE)$ for all $(\alpha_n)_n\in \D^\N$.  Then for every $\eps>0$ there exists $r>0$ so that $\sum_{n}\alpha_na_n$ can be $\eps$-$r$-approximated for all $(\alpha_n)_n\in \D^\N$.
\end{lemma}

\begin{proof}[Sketch of the proof]
The proof follows the proof of  \cite[Lemma 4.9]{BragaFarah2018} almost verbatim. The only difference being that since the unit ball of $\cL(E)$ does not need to be compact in the weak operator topology (indeed, this holds if and only if $E$ is reflexive),  the proof of  \cite[Lemma 4.7]{BragaFarah2018}, which is used in  the proof of \cite[Lemma 4.9]{BragaFarah2018}, does not hold. We present a proof of \cite[Lemma 4.7]{BragaFarah2018} which holds in our setting:

\begin{lemma}\label{Lemma47}
Let $(\N,d)$ be a u.l.f. metric space,  $E$ be a Banach space with an unconditional basis $\cE=(e_n)_n$,  $a\in\cL(E)$, and $\eps,s>0$. If $a$ is not $\eps$-$s$-approximable, then $a\chi_{[1,n]}$ is not  $\eps$-$s$-approximable for all large enough $n\in\N$.
\end{lemma}

\begin{proof}
Without loss of generality, assume $\cE$ is $1$-unconditional. If the lemma fails, then for each $n\in\N$ pick $b_n\in \cL(E)$ with $\propg(b_n)\leq s$ and $\|a\chi_{[1,n]}-b_n\|\leq \eps$ (cf. \cite[Proposition 4.6(i)]{BragaFarah2018}). Clearly,  $\sup_{n}\|b_n\|<\infty$.

\begin{claim}
The sequence $(b_n)_n$ has a subsequence which converges in the strong operator topology.
\end{claim}

\begin{proof}
For each $n\in\N$, let $b_n=[b_{i,j}^n]$ be the matrix representation of $b_n$, i.e., $b_{i,j}^n=e_j^*(b_ne_i)$. Going to a subsequence, assume that  $\lim_nb_{i,j}^n$ exists for all $i,j\in\N$.  Given $\xi=\sum_i\xi(i)e_i\in E$,  $b_n\xi(j)=\sum_{i}b_{i,j}^n\xi(i)$ for all $j\in\N$. Moreover, as $(\N,d)$ is u.l.f. and $\propg(b_n)\leq s$ for all $n\in\N$, there is $k\in\N$ so that for each $j\in\N$, there is a subset $A_j\subset \N$ with $k$ elements so that $b_{i,j}^n= 0$ for all $i\not\in A$. Hence, $\sum_{i}b_{i,j}^n\xi(i)=\sum_{i\in A_j}b_{i,j}^n\xi(i)$ and $\lim_n\sum_{i}b_{i,j}^n\xi(i)$ exists for all $\xi\in E$ and all $j\in\N$. 

We now define a bounded operator $b$ on $E$ by defining  the $j$-th coordinate $b\xi(j)$ for each $\xi\in E$. Precisely,    for  $j\in\N$ and  $\xi=\sum_i\xi(i)e_i\in E$, let $b\xi(j)=\lim_n\sum_{i}b_{i,j}^n\xi(i)$. Notice that, bimonotonicity of the basis implies that
\[
\Big\|\sum_{j=N}^M b\xi(j)e_j\Big\|=\lim_n\Big\| \sum_{j=N}^M\Big(\sum_{i\in A_j}b_{i,j}^n\xi(i)\Big)e_j\Big\|
\leq  \lim_n\Big\| \sum_{j=N}^\infty\Big(\sum_{i\in A_j}b_{i,j}^n\xi(i)\Big)e_j\Big\|.\]
As $d$ is u.l.f., there exists an increasing sequence $(l_N)_N$ so that 
\[ \Big\| \sum_{j=N}^\infty\Big(\sum_{i\in A_j}b_{i,j}^n\xi(i)\Big)e_j\Big\|=\|(\chi_{[N,\infty)}b_n)\xi\|\leq \|b_n\chi_{[l_N,\infty)}\xi\|.\]
This shows that given any $\xi\in E$, the sum  $\sum_{j}b\xi(j)e_j$ converges in $E$, so   $b\xi$ is well defined. Moreover,    $b$ is clearly  bounded and $b_n$ converges to it in the strong operator topology.
\end{proof}

By the previous claim, going to a subsequence if necessary, let  $b=\text{SOT-}\lim_nb_n$.  Clearly, $\propg(b)\leq s$. As $a$ is not $\eps$-$s$-approximable, there exists a unit vector $\xi\in E$ so that  $\|(a-b)\xi\|> \eps$. As $\xi=\lim_n\chi_{[1,n]}\xi$, this contradicts that $\|a\chi_{[1,n]}-b_n\|\leq \eps$ for all $n\in\N$.
\end{proof}

 The rest of the proof follows the proof of \cite[Lemma 4.9]{BragaFarah2018} with Lemma \ref{Lemma47} substituting \cite[Lemma 4.7]{BragaFarah2018}.
\end{proof}

\begin{proof}[Proof of Theorem \ref{ThmCoarseLike}]
Let $\eps, r>0$ and $\cE=(e_n)_n$. Without loss of generality, assume that $\cE$ is normalized. Let 
\[E=\{(n,m)\in \N\times \N\mid d(n,m)\leq r\}.\]
 By Proposition \ref{PropSubsymmetric}, we have that $\sum_{(n,m)\in E}\alpha_{n,m}e_{n,m}$ converges in the strong operator topology to an operator in $\broe(d,\cE)$ for all $(\alpha_{n,m})_{(n,m)\in E}\in \D^E$. 

The hypothesis on $\Phi$ imply that $\sum_{(n,m)\in E}\alpha_{n,m}\Phi(e_{n,m})$ converges in the strong operator topology for all $(\alpha_{n,m})_{(n,m)\in E}\in \D^E$. By Lemma \ref{Lemma49}, there exists $s>0$ so that $\Phi(\sum_{(n,m)\in E}\alpha_{n,m}e_{n,m})=\sum_{(n,m)\in E}\alpha_{n,m}\Phi(e_{n,m})$ can be $\eps$-$s$-approximated for all $(\alpha_{n,m})_{(n,m)\in E}\in \D^E$. Since every operator in $B_{\cL(E)}$ with propogation at most $r$ is of the form $\sum_{(n,m)\in E}\alpha_{n,m}e_{n,m}$ for some $ (\alpha_{n,m})_{(n,m)\in E}\in \D^E$, the result follows. 
\end{proof}

\subsection{Obtaining coarse equivalences}
In Subsection \ref{SubsectionChoosingMaps}, we presented a method of  choosing candidates for being coarse equivalences and in Subsection \ref{SubsectionCoarseLike} we show that some maps between uniform Roe algebras have a coarse-like behavior. In this subsection, we put those together and obtain the desired coarse equivalences.

The next lemma  is a more general version of  \cite[Lemma 3.3]{BragaVignati2019} to our setting. Since its proof translates almost verbatim, we omit it here  (cf. \cite[Lemma 3.4]{BragaFarahVignati2018}). If $(\N,d)$ is a  metric space, $E$ is a Banach space with an unconditional  basis $\cE$,   $a=\sum_n\alpha_ne_{n,n}\in \ell_\infty(\cE)$, and $A\subset \N$, we write 
$a_F=\sum_{n\in A}\alpha_ne_{n,n}$.

\begin{lemma}
Let $(\N,d)$ and $(\N,\partial)$ be metric spaces and $E$ and $F$ be Banach spaces with unconditional bases $\cE$ and $\cF$, respectively. Let $\Phi:\ell_\infty(\cE)\to \broe(\partial,\cF)$ be a strongly continuous linear map. Then for every $b\in \mathrm M_\infty(\cF)$ and every $\eps>0$ there exists a finite $A\subset \N$ so that 
\[\max\{\|b\Phi(a_A)\|,\|\Phi(a_A)b\|\}<\eps\]
for all contractions $a\in \ell_\infty(\cE)$.
\end{lemma}

\begin{theorem}\label{ThmEmbHer} 
Let $(\N,d)$ and $(\N,\partial)$ be u.l.f. metric spaces, and $E$ and $F$ be Banach spaces with bases $\cE$ and $\cF$, respectively. Assume that $\cE$ is $d$-symmetric, $\cF$ is $\partial$-symmetric, and that all ghost idempotents in $\broe(\partial, \cF)$ belong to $\mathrm M_\infty(\cF)$. If $\broe(d,\cE)$ embeds as a Banach algebra onto a hereditary Banach  subalgebra of $ \broe(\partial, \cF)$, then $(\N,d)$ coarsely embeds into $(\N,\partial)$.
\end{theorem}

\begin{proof}
Without loss of generality, assume the bases are normalized and $1$-unconditional.
Let $\Phi:\broe(d,\cE)\to \broe(\partial, \cF)$ be a Banach algebra embedding onto a hereditary subalgebra of $\broe(\partial, \cF)$. By Corollary \ref{CorPickDelta2}, there are $\delta>0$ and maps $f,g:(\N,d)\to (\N, \partial)$ so that 
\[\|e_{g(n),g(n)}\Phi(e_{n,n})e_{f(n),f(n)}\|\geq \delta\]
for all $n\in\N$. By Lemma \ref{LemmaPhiIsStrongContAndRankPres} and Theorem \ref{ThmCoarseLike}, the map $\Phi$ is coarse-like. Hence,  Lemma \ref{LemmaCloseMapsAndUnfFiniteToOne}  implies that $f$ and $g$ are close and we can  fix  $N>0$ so that $\partial(f(n),g(n))<N$ for all $n\in\N$.

Let us notice that $f$ is coarse, for that fix  $r>0$. As  there is a surjective isomorphism $U:E\to \mathrm{Im}(\Phi(1))$ so that $\Phi(a)=UaU^{-1}\Phi(1)$ (Lemma \ref{LemmaPhiIsStrongContAndRankPres}), a simple computation gives that there exists $\theta>0$ so that
\begin{align}\label{EqClaim313}
\|e_{g(n),g(n)}\Phi(e_{m,n}) & e_{f(m),f(m)}\| \\
&> \theta
\|e_{g(n),g(n)}\Phi(e_{n,n})\|\|\Phi(e_{m,m})e_{f(m),f(m)}\|\notag
\end{align}
for all $n,m\in\N$ (see \cite[Claim 3.13]{BragaVignati2019} for similar computation). Therefore, since our choice of $f$ and $g$ imply that \[\|e_{g(n),g(n)}\Phi(e_{n,n})\|\geq \delta \ \text{ and }\ \|\Phi(e_{n,n})e_{f(n),f(n)}\|\geq \delta,\]
 we have that $\|e_{g(n),g(n)}\Phi(e_{m,n})e_{f(m),f(m)}\|> \theta^2\delta^2$ for all $n,m\in\N$. 

As $\Phi$ is coarse-like, there exists $s>0$ so that $\Phi(e_{n,m})$ is $\theta^2\delta^2$-$s$-approximable for all $n,m\in\N$ with $d(n,m)<r$. So, $\|e_{k,k}\Phi(e_{n,m})e_{\ell,\ell}\|\leq \theta^2\delta^2$ if $\partial(k,\ell)>s$ and $d(n,m)<r$. Therefore, $\partial (f(n),g(m))\leq s$ for all $n,m\in\N$ with $d(n,m)<r$, which implies that $\partial(f(n),f(m))\leq r+N$ for all $n,m\in\N$ with $d(n,m)<r$.
 
We now follow  \cite[Lemma 3.14]{BragaVignati2019} and sketch the proof that $f$ is expanding. Suppose $f$ is not expanding, so there is $s>0$ and sequences $(x^n_1)_n$ and 
 $(x^n_2)_n$ in $\N$ so that $d(x^n_1,x^n_2)\geq n$ and $\partial (f(x^n_1),f(x^n_2))\leq s$ for all $n\in\N$. As  $d(x^n_1,x^n_2)\geq n$ for all $n\in\N$, by going to a subsequence, we can assume $(x^n_1)_n$ is a sequence of distinct elements. Moreover, as $\partial (f(x^n_1),f(x^n_2))\leq s$ for all $n\in\N$, proceeding as in \cite[Claim 3.15]{BragaVignati2019}, we can go to a further subsequence and assume that both $(f(x^n_1))_n$ and $(f(x^n_2))_n$ are sequences of distinct elements. 
 
 As $\cF$ is $\partial$-subsymmetric and $\partial (f(x^n_1),f(x^n_2))\leq s$ for all $n\in\N$, the sum $\sum_{n}e_{f(x^n_1),f(x^n_2)}$ converges in the operator topology to an operator with finite propagation, hence in $\broe(\partial,\cF)$. Therefore, as $\Phi(\broe(d, \cE))$ is a hereditary Banach subalgebra of $\broe(\partial, \cF)$, there is $a\in \broe(d,\cE)$ so that 
 \[\Phi(a)=\Phi(1)\Big(\sum_{n\in \N}e_{f(x^n_1),f(x^n_2)}\Big)\Phi(1).\]
Proceeding as in \cite[Claim 3.16]{BragaVignati2019}, it follows that $\inf_n\|e_{x^n_2,x^n_2}ae_{x^n_1,x^n_1}\|>0$. This gives a contradiction since $\lim_nd(x^n_1,x^n_2)=\infty$. So $f$ is expanding and we are done. 
\end{proof}

\begin{theorem}\label{ThmIsomorBanAlgMetricSym}
Let $(\N,d)$ and $(\N,\partial)$ be u.l.f. metric spaces, and $E$ and $F$ be Banach spaces with $d$-symmetric basis $\cE$ and $\partial$-symmetric  basis $\cF$, respectively. Assume that all ghost idempotents  in $\broe(d, \cE)$ and $\broe(\partial, \cF)$ are compact. If $\broe(d,\N)$ and  $ \broe(\partial, \N)$ are  Banach algebra isomorphic, then $(\N,d)$ and  $(\N,\partial)$ are coarsely equivalent.
\end{theorem}

\begin{proof}
 Let $\Phi:\broe(d,\cE)\to \broe(\partial, \cF)$ be a Banach algebra isomorphism. Proceeding as in the proof of Theorem \ref{ThmEmbHer}, there are coarse maps $f:(\N,d)\to (\N,\partial)$ and $h:(\N,\partial)\to (\N,d)$ so that 
 \[\|\Phi(e_{n,n})e_{f(n),f(n)}\|\geq \delta \ \text{ and }\  \|\Phi^{-1}(e_{n,n})e_{h(n),h(n)}\|\geq \delta\]
 for all $n\in\N$. We only need to notice that $f\circ h$ and $h\circ f$ are close to $\mathrm{Id}_{(\N,\partial)}$ and $\mathrm{Id}_{(\N,d)}$, respectively. 
 
 Suppose $h\circ f$ is not closed to $\mathrm{Id}_{(\N,d)}$. Then there exists a sequence $(n_k)_k$ in $(\N, d)$ so that $d(n_k,h(f(n_k)))>k$. For brevity, let $m_k=f(n_k)$ and $\ell_k=h(m_k)$ for all $k\in\N$. By \eqref{EqClaim313} above, there exists $\theta>0$ so that 
 $\|e_{n_k,n_k}\Phi^{-1}(e_{m_k,m_k})e_{\ell_k,\ell_k}\|\geq \theta$ for all $k\in\N$ (cf. \cite[Claim 3.13]{BragaVignati2019}). Since $\lim_kd(n_k,\ell_k)=\infty$, this contradicts the fact that $\Phi^{-1}$ is coarse-like. Hence $h\circ f$ is close to $\mathrm{Id}_{(\N,d)}$. The proof that $f\circ h$ is close to $\mathrm{Id}_{(\N,\partial)}$ is completely analogous, so we are done.
\end{proof}

\begin{proof}[Proof of Theorem \ref{ThmIsomorBanAlg}]
This is a particular case of Theorem \ref{ThmIsomorBanAlgMetricSym}.
\end{proof}

\subsection{Results for spaces with property A}\label{SubsectionPropARig}
By Theorem \ref{ThmPropertyAGhostsAreComp}, it is clear that the infimum in  Corollary \ref{CorPickDelta2} is greater than zero as long as the embedding $\Phi:\broe(d,\cE)\to \broe(\partial,\cF)$ satisfies $\Phi(\broe[d,\cE])\subset \broe^r(\partial,\cF)$.

\begin{theorem}\label{ThmEMBBanAlgPropAINTRO}
Let $ d$ and $\partial$ be uniformly locally finite metrics on $\N$,  and $E$ and $F$ be Banach spaces with $1$-unconditional bases $\cE$ and $\cF$ which are $d$-symmetric and $\partial$-symmetric, respectively. If $(\N, \partial)$ has property A and there is a Banach algebra embedding $\Phi:\broe(d,\N)\to  \broe(\partial, \N)$ onto a hereditary subalgebra of $\broe(\partial, \N)$ so that $\Phi(\broe[d,\cE])\subset \broe^r(\partial,\cF)$, then $(\N,d)$ coarsely embeds into   $(\N,\partial)$.
\end{theorem}

\begin{proof}[Proof of Theorem \ref{ThmIsomorBanAlgPropAINTRO} and Theorem \ref{ThmEMBBanAlgPropAINTRO}]
Since  the  infimum in  Corollary \ref{CorPickDelta2} is greater than zero,   the proofs of Theorem \ref{ThmEmbHer} and Theorem \ref{ThmIsomorBanAlgMetricSym} give us the desired results.
\end{proof}

As mentioned in the introduction, we can obtain a slightly stronger result. Precisely, we do not need to require that  $\Phi(\broe[d,\cE])\subset \broe^r(\partial,\cF)$ in order for the  infimum in  Corollary \ref{CorPickDelta2} to be greater than zero, but only that \[\Phi(\broe[d,\cE])\subset \broe^r(\partial,\cF)\cup \cL(E)^\complement_r,\] i.e., if $a$ is band-$r$-dominated and $\Phi(a)$ is regular, than $\Phi(a)$ is band-$r$-dominated.\footnote{We do not know if this is something automatic, see Problem \ref{ProbAutomatic}.} For that, we need the following complement to Proposition \ref{PropSOTConvSeriesRegNew}.

\begin{proposition}\label{PropSOTConvSeriesReg}
 Let $(\N,\partial)$ be a u.l.f. metric space and   $F$ be a Banach space with $1$-unconditional basis  $\cF$. Let $(p_n)_n$ be a sequence of rank 1 operators so that $\mathrm{SOT}\text{-}\sum_{n\in\M}p_n\in \broe(\partial,\cF)$ for all $\M\subset \N$ and $\inf_n\|p_n\|>0$. Given an infinite $\M_0\subset \N$  there exist an infinite $\M\subset \M_0$  so that 
   $\mathrm{SOT}\text{-} \sum_{n\in\M}p_n$ is regular.
\end{proposition}

\begin{proof}
For each $n\in\N$, pick a unit vector $\xi_n\in F$ and $f_n\in F^*$ so that  $\|f_n\|=\|p_n\|$  and $p_n=\xi_n\otimes f_n$. Since  $\delta=\inf_n\|p_n\|>0$,  $(f_n)_n$ is semi-normalized. As  $\sum_{n\in\N}p_n$ converges in the strong operator topology, the sequence $(f_n)_n$ must be weak$^*$-null. Since $p_n\in \broe(\partial, \cF)$, each $p_n$ can be approximated arbitrarily well by band operators. Hence, it follows that \[\lim_{m\to \infty}\chi_{[m,\infty]}f_n=0.\]
Therefore, going to a subsequence if necessary, we can pick a   sequence $(g_n)_n$ in $F^*$ which is a block subsequence of $(e^*_n)_n$ and  so that $\|f_n-g_n\|<2^{-n}$ for all $n\in\N$. Clearly, $\mathrm{SOT}\text{-}\sum_{n\in\M}\xi_n\otimes g_n$ converges to an operator in $\broe(\partial, \cF)$ for all $\M\subset \N$. This implies  that $\lim_n\xi_n(j)=0$ for all $j\in\N$. Indeed, this follows since, for each $j\in\N$, we have that 
\[\chi_{\{j\}}\Big(\sum_{j\in\M}\xi_j\otimes g_j\Big)\chi_{\supp(g_n)}=\xi_n(j)e_j\otimes g_n\]
and $\lim_n\partial (\{j\},\supp(g_n))=\infty$. Hence, going to a further subsequence, pick a sequence $(\zeta_n)_n$ in $B_F$ which is a block subsequence of $(e_n)_n$ and so that $\|\xi_n-\zeta_n\|<2^{-n}$ for all $n\in\N$. Clearly, $\mathrm{SOT}\text{-}\sum_{n\in\M}\zeta_n\otimes g_n$ converges to an operator in $\broe(\partial, \cF)$ for all $\M\subset \N$.

Since $(\zeta_n)_n$ and $(g_n)_n$ are block sequences in $(e_n)_n$ and $(e^*_n)_n$, respectively, the unconditionality of $(e_n)_n$ implies that $\sum_{n\in\N}|\zeta_n|\otimes|g_n|$ is a well defined positive linear function on $E$. Hence, by Proposition \ref{PropPositiveAreBounded}, $\sum_{n\in\N}|\zeta_n|\otimes|g_n|$ is bounded and we must have that $|\sum_{n\in\N}\zeta_n\otimes g_n|$ exists and equals $\sum_{n\in\N}|\zeta_n|\otimes|g_n|$.

In order to conclude that $\sum_{n\in\N}p_n$ is regular, notice that, letting $K=\max\{1,\sup_n\|f_n\|\}<\infty$, we have
\begin{align*}
\|p_n-\zeta_n\otimes g_n\|_r& \leq \|(\xi_n-\zeta_n)\otimes f_n\|_r +\|\zeta_n\otimes (f_n-g_n)\|_r\\
&\leq K\||\xi_n-\zeta_n|\|+\||f_n-g_n|\|\\
&\leq 2^{-n+1}K
\end{align*}
for all $n\in\N$. Hence $\sum_n(p_n-\zeta_n\otimes g_n)$ is an absolutely $\|\cdot\|_r$-converging series in $\broe^r(\partial, \cF)$, so $\sum_n(p_n-\zeta_n\otimes g_n)\in \broe^r(\partial,\cF)$. Since \[\sum_np_n=\sum_n(p_n-\zeta_n\otimes g_n)+\sum_n\zeta_n\otimes g_n,\] this shows that $\sum_{n\in\N}p_n$ is regular.
\end{proof}

Corollary \ref{CorPickDelta2} then becomes:

\begin{corollary}\label{CorPickDelta2new}
Let $(\N,d)$ and $(\N,\partial)$ be metric spaces, $E$ and $F$ be Banach spaces with $1$-unconditional basis $\cE$ and $\cF$, and  let $\Phi:\broe(d,\cE)\to \broe(\partial,\cF)$ be an embedding onto a hereditary subalgebra of $\broe(\partial,\cF)$ so that  \[\Phi(\broe[d,\cE])\subset \broe^r(\partial,\cF)\cup \cL(F)_r^\complement.\]   If  $(\N,\partial)$  has property A, then   
\[\delta=\inf_{n\in\N }\sup_{m,k\in\N}\|e_{m,m}\phi(e_{n,n})e_{k,k}\|>0.\]
\end{corollary}

\begin{proof}
This follows straightforwardly from Lemma \ref{LemmaPickDelta}, Lemma \ref{LemmaPhiIsStrongContAndRankPres}, Proposition \ref{PropSOTConvSeriesRegNew} and  Theorem \ref{ThmPropertyAGhostsAreComp}. 
\end{proof}

We finish this section stating the consequences of Corollary \ref{CorPickDelta2new} and the proofs of Theorem  \ref{ThmEmbHer} and Theorem \ref{ThmIsomorBanAlgMetricSym}.

\begin{theorem}\label{ThmEmbHerPropA}
Let $ d$ and $\partial$ be u.l.f.  metrics on $\N$,  and $E$ and $F$ be Banach spaces with $1$-unconditional bases $\cE$ and $\cF$ which are  $d$-symmetric and   $\partial$-symmetric, respectively. If $(\N, \partial)$ has property A and there is a Banach algebra embedding $\Phi:\broe(d,\N)\to  \broe(\partial, \N)$ onto a hereditary subalgebra of $\broe(\partial, \N)$ so that 
 \[\Phi(\broe[d,\cE])\subset \broe^r(\partial,\cF)\cup \cL(F)_r^\complement,\]
 then $(\N,d)$ coarsely embeds into   $(\N,\partial)$.\qed
\end{theorem}

\begin{theorem}\label{ThmIsomorBanAlgPropA}
Let $ d$ and $\partial$ be u.l.f.  metrics on $\N$,  and $E$ and $F$ be Banach spaces with $1$-unconditional bases $\cE$ and $\cF$ which are $d$-symmetric and $\partial$-symmetric, respectively. If $(\N, \partial)$ has property A and there is a Banach algebra isomorphism  so that 
 \[\Phi(\broe[d,\cE])\subset \broe^r(\partial,\cF)\cup \cL(F)_r^\complement\text{ and }\Phi^{-1}(\broe[\partial,\cF])\subset \broe^r(d,\cE)\cup \cL(E)_r^\complement,\]
 then $(\N,d)$ and   $(\N,\partial)$ are coarsely equivalent.\qed
\end{theorem}

\section{Open questions}\label{SectionOpenProb}
We now list a couple of natural questions which this paper leaves open. For instance, Theorem \ref{ThmBanachIsometryOrderIso} gives that, for $1$-symmetric basis, the existence of an order Banach space isomorphism which is also an isometry between uniform Roe algebras gives bijective coarsely equivalence between the base metric spaces. We ask whether Banach space isometry  would be already enough for that.

\begin{problem}\label{Prob2}
Let $(\N,d)$ and $(\N,\partial)$ be u.l.f. metric spaces, and $E$ and $F$ be strictly  convex Banach spaces with $1$-symmetric  bases $\cE$ and $\cF$, respectively. If  $\broe(d,\cE)$ and $\broe(\partial,\cF)$ are isometric,  does it follow that  $(\N,d)$ and $(\N,\partial)$ are coarsely equivalent?  
\end{problem}

We do not know the answer to Problem \ref{Prob2} even for $X=\ell_2$ and $\cE$ the standard basis of $\ell_2$. In fact, we do not even have an answer for the following:

\begin{problem}\label{ProblemIsometryRig}
Let $X$ and $Y$ be u.l.f. metric spaces and suppose that 
\begin{enumerate}
    \item $\cstu(X)$ and $\cstu(Y)$ are isometric  as Banach spaces, and
    \item $\cstu(X)$ and $\cstu(Y)$ are isomorphic as ordered Banach spaces.
\end{enumerate}
 Does it follow that $X$ and $Y$ are coarsely equivalent?
\end{problem}

Let $\cE$ be a shrinking basis for the Banach space complexification of Argyros-Haydon space. By Proposition \ref{PropArgyrosHaydon}, the set \[\{\broe(d,\cE)\mid d\text{ is a metric on }\N\}\] contains only $1$ element. Can we obtain similar results for an arbitrary $n\in\N$? 

\begin{problem}
What are the possible cardinals   $\kappa$ so that  there is a Banach space $E$ with a basis $\cE$  so that there are exactly $\kappa$ possibilities for $\cB(d,\cE)$? 
\end{problem}

Proposition \ref{PropAlvaroFarmer} shows that nonisomorphic Banach spaces can have Banach space isomorphic uniform Roe algebras. However, we do not have an example of such Banach spaces with Banach algebra isomorphic uniform Roe algebras:

\begin{problem}\label{Prob4}
Are there nonisomorphic Banach spaces $E$ and $F$ with unconditional basis $\cE$ and $\cF$, respectively, so that   $\broe(d,\cE)$ and $\broe(d,\cF)$ are Banach algebra isomorphic? 
\end{problem}

At last,  the condition $\Phi(\broe^r(d,\cE))\subset \broe^r(\partial,\cF)\cup \cL(F)_r^\complement$ in Theorem \ref{ThmEmbHerPropA} and Theorem \ref{ThmIsomorBanAlgPropA} is rather technical, and we would like to get rid of it. We actually do not know when a regular operator $\Phi(a)$ must belong to the regular uniform Roe algebra $\broe^r(\partial,\cF)$ --- notice that it would  enough to assume that $a$ is idempotent for our goals.

 \begin{problem}\label{ProbAutomatic}
 Let $(\N,d)$ and $(\N,\partial)$ be u.l.f. metric spaces and let $E$ and $F$ be Banach spaces with $1$-symmetric bases $\cE$ and $\cF$, respectively. Let $\Phi:\broe(d,\cE)\to \broe(\partial, \cF)$ be a Banach algebra isomorphism  and    $a\in \broe^r(d,\cE)$ be so that $\Phi(a)$ is regular. When can we say that  $\Phi(a)\in \broe^r(\partial, \cF)$?
 \end{problem}

\section{Appendix}\label{SectionApp}

In this section, we review the basics of real Banach lattices. We recall the definition and basic properties of complex Banach lattices in Section \ref{SectionRegularURA}. For a detailed treatment of Banach lattices,   we refer to the monographs \cite{Meyer-NiebergBook1991} and \cite{SchaeferBook1974}.

\begin{definition}\label{DefiBanachLattice}
A real Banach space $E$ with a partial order $\leq$ is an \emph{ordered real Banach space} if 
\begin{enumerate}
\item $x\leq y$ implies $x+z\leq y+z$ for all $z\in X$,
\item $x\geq 0$ implies $\lambda x\geq 0$ for all $\lambda\geq 0$, 
\end{enumerate}
The set $E_+=\{x\in X\mid x\geq 0\}$ is the set of \emph{positive} elements of $E$ and the set $E_r=\{x-y\in E\mid x,y\in E_+\}$ is the set of \emph{regular} elements of $E$.

An ordered real Banach space $E$ is a \emph{real Banach lattice} if \begin{enumerate}\setcounter{enumi}{2}
\item for all  $x,y\in E$, there exist a greatest lower bound $x\wedge y$ and a least upper bound $x\vee y$, and
\item $|x|\leq |y|$ implies $\|x\|\leq\|y\|$ for all $x,y\in E$, where $|x|\coloneqq x\vee (-x)$.
\end{enumerate}
\end{definition}

We notice that the lattice operations are all norm continuous in any real Banach lattice \cite[Proposition 1.1.6]{Meyer-NiebergBook1991}.  Every real Banach space $E$ with an $1$-unconditional basis $\cE=(e_n)_n$ is a real Banach lattice with the partial order given coordinate-wise order, i.e.,  $\sum_n\lambda_ne_n\leq \sum_n \theta_ne_n$ if and only if $\lambda_n\leq \theta_n$ for all $n\in\N$. Given $x=\sum_n\lambda_ne_n\in E$, this partial order gives that $|x|=\sum_n|\lambda_n|e_n$.

 Given a real Banach lattice $E$,   $x_1,\ldots, x_n\in E$, and $p,q\in (1,\infty)$ with $1/p+1/q=1$, define\footnote{For Dedekind complete Banach lattices this supremum is clearly well defined. For the general case see  \cite[Page 25]{JohnsonLindenstrauss2001Handbook}.} 
\[\Big(\sum_{i}^n|x_i|^p\Big)^{1/p}=\sup\Big\{\sum_{i=1}^n\alpha_ix_i\mid (\alpha_i)_{i=1}^n\in \R^n,\ \sum_{i=1}^n|\alpha_i|^q\leq 1\Big\}.\]

\begin{proposition}\cite[Proposition 1.d.2(iii)]{LindenstraussTzafririVol2}
Let $E$ be a Banach lattice, $x_1,\ldots,x_n\in E$ and $x_1^*,\ldots, x_n^*\in E^*$. Then, given $p,q\in (1,\infty)$ with $1/p+1/q=1$, we have\footnote{This formula represents the evaluation of the functional $(\sum_{j=1}^n|x_j^*|^p )^{1/p}$ at the vector $(\sum_{j=1}^n|x_j|^q)^{1/q}$.}
\[\sum_{j=1}^nx_j^*x_j\leq \Big(\sum_{j=1}^n|x_j^*|^p\Big)^{1/p}\Big(\sum_{j=1}^n|x_j|^p\Big)^{1/p}.\]\label{PropHolderForLattices}
\end{proposition}

Let $E$  and $F$ be real Banach lattices. A linear map $a:E\to F$ is called \emph{positive} if $a(E_+)\subset F_+$ and we write $a\geq 0$. This  defines a canonical order on  $\cL(E,F)$. A bounded linear map $E\to F$ is \emph{regular} if it is the difference between positive linear maps.  If $F$ is \emph{Dedekind complete} (i.e., if every order bounded subset of it has a supremum and an infimum), then the space of all regular operators $E\to F$  is a Dedekind complete vector lattice \cite[Theorem 1.3.2]{Meyer-NiebergBook1991}. We denote this vector lattice  by $\cL(E,F)_r$  and the subset of all positive operators by $\cL(E,F)_+$. Notice that this is consistent with both the $+$-index and the  $r$-index notation in Definition \ref{DefiBanachLattice} since $(\cL(E,F)_r)_r=\cL(E,F)_r$ and $(\cL(E,F)_r)_+=\cL(E,F)_+$. Given  $a\in \cL(E,F)_r$,  the \emph{$r$-norm} of $a$ is defined by 
    \[\|a\|_r=\inf\Big\{\|b\|\mid b\in \cL(E,F)_+, |a\xi|\leq b|\xi|\  \text{ for all }\ \xi\in E_+\Big\}.\] 
Dedekind completeness of $F$ implies that $(L(E,F)_r,\|\cdot\|_r)$ is a Banach lattice and that $\|a\|_r=\||a|\|$ for all $a\in \cL(E,F)_r$ \cite[Proposition 1.3.6]{Meyer-NiebergBook1991}. 

\begin{example}
Although $\cL(E)_r$ may equal $\cL(E)$ --- e.g., if $E$ is a Dedekind complete AM-space with unit (see \cite[Theorem 1.5]{SchaeferBook1974} for details) ---, this is a rare phenomenon. In  fact, one can easily find $a\in \cstu(X)\setminus \cL(\ell_2(X))_r$ for any   u.l.f. metric space $X$ with infinitely many points. Indeed, write $X=\bigsqcup_{n\in\N}X_n$, with $|X_n|=2^n$ for each $n\in\N$,  and for each $n\in\N$ let $a_n\in \cL(\ell_2(X_n))$ be a norm $1$ operator so that $\|a_n\|_r=2^{n/2}$. The sum $a=\sum_{n\in\N}2^ {-n/3}a_n$ defines  a compact operator (hence, $a\in \cstu(X)$) which is not regular. We refer the reader to   \cite[Chapter IV, \S 1, Example 2]{SchaeferBook1974} for details.
\end{example}

\begin{proposition}\cite[Proposition 1.3.5]{Meyer-NiebergBook1991}
Let $E$ be a Banach lattice. Then every positive linear map  $a: E\to E$ is bounded.\label{PropPositiveAreBounded}
\end{proposition}

\begin{acknowledgments}
I would like to thank Bill Johnson for  bringing the papers   \cite{AlvaroFarmer1996,Eidelheit1940} to my attention. I also thank Vladimir  Troitsky for valuable conversations.
\end{acknowledgments}


\begin{thebibliography}{10}

\bibitem{AlvaroFarmer1996}
A.~Alvaro and J.~Farmer.
\newblock On the structure of tensor products of $l_p$-spaces.
\newblock {\em Pacific J. Math.}, 175:13--37, 1996.

\bibitem{ArgyrosHaydon2011}
S.~Argyros and R.~Haydon.
\newblock A hereditarily indecomposable $\mathscr{L}_\infty$-space that solves
  the scalar-plus-compact problem.
\newblock {\em Acta Math.}, 206(1):1--54, 2011.

\bibitem{BragaChungLi2019}
B.~M. {Braga}, Y.~{Chung}, and K.~{Li}.
\newblock {Coarse Baum-Connes conjecture and rigidity for Roe algebras}.
\newblock arXiv:1907.10237.

\bibitem{BragaFarah2018}
B.~M. {Braga} and I.~{Farah}.
\newblock {On the rigidity of uniform Roe algebras over uniformly locally
  finite coarse spaces}.
\newblock arXiv:1805.04236.

\bibitem{BragaFarahVignati2019}
B.~M. {Braga}, I.~{Farah}, and A.~{Vignati}.
\newblock {Embeddings of uniform Roe algebras}.
\newblock arXiv:1904.07291.

\bibitem{BragaFarahVignati2020}
B.~M. Braga, I.~Farah, and A.~Vignati.
\newblock {General uniform Roe algebra rigidity}.
\newblock {\em arXiv e-prints}.

\bibitem{BragaFarahVignati2018}
B.~M. {Braga}, I.~{Farah}, and A.~{Vignati}.
\newblock {Uniform Roe coronas}.
\newblock arXiv:1810.07789.

\bibitem{BragaVignati2019}
B.~M. Braga and A.~Vignati.
\newblock {On the uniform Roe algebra as a Banach algebra and embeddings of
  $\ell_p$ uniform Roe algebras}.
\newblock {\em arXiv e-prints}, page arXiv:1906.11725, Jun 2019.

\bibitem{Chandler-WildeLindner2008JFA}
S.~N. Chandler-Wilde and M.~Lindner.
\newblock Sufficiency of {F}avard's condition for a class of band-dominated
  operators on the axis.
\newblock {\em J. Funct. Anal.}, 254(4):1146--1159, 2008.

\bibitem{ChungLi2018}
Y.~Chung and K.~Li.
\newblock Rigidity of {$\ell^p$} {R}oe-type algebras.
\newblock {\em Bull. Lond. Math. Soc.}, 50(6):1056--1070, 2018.

\bibitem{Eidelheit1940}
M.~Eidelheit.
\newblock On isomorphisms of rings of linear operators.
\newblock {\em Studia Mathematica}, 9:97--105, 1940.

\bibitem{EwertMeyer2019}
E.~Ewert and R.~Meyer.
\newblock Coarse {G}eometry and {T}opological {P}hases.
\newblock {\em Comm. Math. Phys.}, 366(3):1069--1098, 2019.

\bibitem{FigielJohnson1974}
T.~Figiel and W.~B. Johnson.
\newblock A uniformly convex {B}anach space which contains no {$l_{p}$}.
\newblock {\em Compositio Math.}, 29:179--190, 1974.

\bibitem{higson2002counterexamples}
N.~Higson, V.~Lafforgue, and G.~Skandalis.
\newblock Counterexamples to the {B}aum--{C}onnes conjecture.
\newblock {\em Geom. Funct. Anal.}, 12(2):330--354, 2002.

\bibitem{HigsonRoe1995}
N.~Higson and J.~Roe.
\newblock On the coarse {B}aum-{C}onnes conjecture.
\newblock In {\em Novikov conjectures, index theorems and rigidity, {V}ol.\ 2
  ({O}berwolfach, 1993)}, volume 227 of {\em London Math. Soc. Lecture Note
  Ser.}, pages 227--254. Cambridge Univ. Press, Cambridge, 1995.

\bibitem{JohnsonLindenstrauss2001Handbook}
W.~Johnson and J.~Lindenstrauss.
\newblock Basic concepts in the geometry of {B}anach spaces.
\newblock In {\em Handbook of the geometry of {B}anach spaces, {V}ol. {I}},
  pages 1--84. North-Holland, Amsterdam, 2001.

\bibitem{JohnsonRosenthalZippin1971}
W.~B. Johnson, H.~P. Rosenthal, and M.~Zippin.
\newblock On bases, finite dimensional decompositions and weaker structures in
  {B}anach spaces.
\newblock {\em Israel J. Math.}, 9:488--506, 1971.

\bibitem{Kubota2017}
Y.~Kubota.
\newblock Controlled topological phases and bulk-edge correspondence.
\newblock {\em Comm. Math. Phys.}, 349(2):493--525, 2017.

\bibitem{LindenstraussTzafririBookVol1}
J.~Lindenstrauss and L.~Tzafriri.
\newblock {\em Classical {B}anach spaces. {I}}.
\newblock Springer-Verlag, Berlin-New York, 1977.
\newblock Sequence spaces, Ergebnisse der Mathematik und ihrer Grenzgebiete,
  Vol. 92.

\bibitem{LindenstraussTzafririVol2}
J.~Lindenstrauss and L.~Tzafriri.
\newblock {\em Classical {B}anach spaces. {II}}, volume~97 of {\em Ergebnisse
  der Mathematik und ihrer Grenzgebiete [Results in Mathematics and Related
  Areas]}.
\newblock Springer-Verlag, Berlin-New York, 1979.
\newblock Function spaces.

\bibitem{Meyer-NiebergBook1991}
P.~Meyer-Nieberg.
\newblock {\em Banach lattices}.
\newblock Universitext. Springer-Verlag, Berlin, 1991.

\bibitem{MunozSarantopoulosAndrew1999Studia}
G.~A. Mu\~{n}oz, Y.~Sarantopoulos, and A.~Tonge.
\newblock Complexifications of real {B}anach spaces, polynomials and
  multilinear maps.
\newblock {\em Studia Math.}, 134(1):1--33, 1999.

\bibitem{RabinovichRochRoe2004IntEpOpTh}
V.~S. Rabinovich, S.~Roch, and J.~Roe.
\newblock Fredholm indices of band-dominated operators.
\newblock {\em Integral Equations Operator Theory}, 49(2):221--238, 2004.

\bibitem{RabinovichRochSilbermann1998IntEqOpTh}
V.~S. Rabinovich, S.~Roch, and B.~Silbermann.
\newblock Fredholm theory and finite section method for band-dominated
  operators.
\newblock volume~30, pages 452--495. 1998.
\newblock Dedicated to the memory of Mark Grigorievich Krein (1907--1989).

\bibitem{RabinovichRochSilbermann2008}
V.~S. Rabinovich, S.~Roch, and B.~Silbermann.
\newblock On finite sections of band-dominated operators.
\newblock In {\em Operator algebras, operator theory and applications}, volume
  181 of {\em Oper. Theory Adv. Appl.}, pages 385--391. Birkh\"{a}user Verlag,
  Basel, 2008.

\bibitem{RoeBook}
J.~Roe.
\newblock {\em Lectures on coarse geometry}, volume~31 of {\em University
  Lecture Series}.
\newblock American Mathematical Society, Providence, RI, 2003.

\bibitem{RoeWillett2014}
J.~Roe and R.~Willett.
\newblock Ghostbusting and property {A}.
\newblock {\em J. Funct. Anal.}, 266(3):1674--1684, 2014.

\bibitem{SchaeferBook1974}
H.~H. Schaefer.
\newblock {\em Banach lattices and positive operators}.
\newblock Springer-Verlag, New York-Heidelberg, 1974.
\newblock Die Grundlehren der mathematischen Wissenschaften, Band 215.

\bibitem{Seidel2014LinAlgAppl}
M.~Seidel.
\newblock Fredholm theory for band-dominated and related operators: a survey.
\newblock {\em Linear Algebra Appl.}, 445:373--394, 2014.

\bibitem{SpakulaWillett2013}
J.~{\v{S}}pakula and R.~Willett.
\newblock On rigidity of {R}oe algebras.
\newblock {\em Adv. Math.}, 249:289--310, 2013.

\bibitem{SpakulaWillett2017}
J.~\v{S}pakula and R.~Willett.
\newblock A metric approach to limit operators.
\newblock {\em Trans. Amer. Math. Soc.}, 369(1):263--308, 2017.

\bibitem{WhiteWillett2017}
S.~{White} and R.~Willett.
\newblock {Cartan subalgebras of uniform Roe algebras}.
\newblock {\em ArXiv:1808.04410}, August 2018.

\bibitem{Willett2009}
R.~Willett.
\newblock Some notes on property {A}.
\newblock In {\em Limits of graphs in group theory and computer science}, pages
  191--281. EPFL Press, Lausanne, 2009.

\bibitem{Willett2011IntEqOpTh}
R.~Willett.
\newblock An index theorem for band-dominated operators with slowly oscillating
  coefficients (after {D}eundyak and {S}hteinberg).
\newblock {\em Integral Equations Operator Theory}, 69(3):301--316, 2011.

\bibitem{willettYuBook}
R.~Willett and G.~Yu.
\newblock Higher index theory.
\newblock Book draft, 2019.

\bibitem{Yu2000}
G.~Yu.
\newblock The coarse {B}aum-{C}onnes conjecture for spaces which admit a
  uniform embedding into {H}ilbert space.
\newblock {\em Invent. Math.}, 139(1):201--240, 2000.

\end{thebibliography}
\end{document}